\documentclass[a4paper]{amsart}

\usepackage[headings]{fullpage}
\usepackage{tikz}
\usepackage[alphabetic]{amsrefs}
\usepackage{color,hyperref}
 
\theoremstyle{plain}
\newtheorem{theorem}{Theorem}[section]
\newtheorem{lemma}[theorem]{Lemma}

\newtheorem{coro}[theorem]{Corollary}

\theoremstyle{definition}

\newtheorem{example}[theorem]{Example}

\theoremstyle{remark}
\newtheorem{remark}[theorem]{Remark}

\numberwithin{equation}{section}
\newcommand{\abs}[1]{\lvert#1\rvert}
\newcommand{\labs}[1]{\left\lvert\,#1\,\right\rvert}

\newcommand{\Lr}[1]{\left(#1\right)}
\newcommand{\lr}[1]{\Bigl(#1\Bigr)}
\newcommand{\set}[2]{\left\{\,#1\,\mid\,#2\,\right\}}
\newcommand{\diff}[2]{\dfrac{\partial #1}{\partial #2}}
\newcommand{\nm}[2]{\|\,#1\,\|_{#2}}
\newcommand{\jump}[1]{[\![#1]\!]}

\newcommand{\wnm}[1]{|\!|\!|#1|\!|\!|_{\iota,h}}
\newcommand{\inm}[1]{\|#1\|_{\iota}}

\newcommand{\mc}[1]{\mathcal{#1}}
\newcommand{\mb}[1]{\mathbb{#1}}

\newcommand{\wh}[1]{\widehat{#1}}
\newcommand{\wt}[1]{\widetilde{#1}}

\def\al{\alpha}

\def\del{\delta}

\def\na{\nabla}
\def\eps{\epsilon}
\def\pa{\partial}
\def\Om{\Omega}

\def\lam{\lambda}
\def\x{\times}

\def\mr{\mathrm}
\def\md{\mathrm{d}}

\def\C{\mathbb C}
\def\D{\mathbb D}

\def\T{\mc{T}}

\def\dx{\,\mathrm{d}\boldsymbol x}

\def\dsx{\,\md\sigma(\boldsymbol x)}

\DeclareMathOperator{\divop}{div}

\def\negint{{\int\negthickspace\negthickspace\negthickspace
\negthinspace -}}
\newcommand{\nn}{\nonumber}
\begin{document}
\title[Taylor-Hood like FEMs for nearly incompressible SGE problems]{Taylor-Hood like finite elements for nearly incompressible strain gradient elasticity problems}
\author[Y. L. Liao\and P. B. Ming]{Yulei Liao\and Pingbing Ming}
\address{LSEC, Institute of Computational Mathematics and Scientific/Engineering Computing, AMSS, Chinese Academy of Sciences, No. 55, East Road Zhong-Guan-Cun, Beijing 100190, China}
\address{School of Mathematical Sciences, University of Chinese Academy of Sciences, Beijing 100049, China}
\email{liaoyulei@lsec.cc.ac.cn, mpb@lsec.cc.ac.cn}

\author[Y. Xu]{Yun Xu}
\address{Laboratory of Computational Physics, Institute of Applied Physics and Computational Mathematics, Beijing 100088, China}
\email{xu\_yun@iapcm.ac.cn}
\thanks{The work of Liao and Ming were supported by National Natural Science Foundation of China through Grant No. 11971467. The work of Xu was supported by National Natural Science Foundation of China through Grant No. 11772067.}
\begin{abstract}

We propose a family of mixed finite elements that are robust for the nearly incompressible strain gradient model, which is a fourth-order singular perturbed elliptic system. The element is similar to [C. Taylor and P. Hood, {\it Comput. \& Fluids}, {\bf 1}(1973), 73-100] in the Stokes flow. Using a uniform discrete B-B inequality for the mixed finite element pairs, we show the optimal rate of convergence that is robust in the incompressible limit. By a new regularity result that is uniform in both the materials parameter and the incompressibility, we prove the method converges with $1/2$ order to the solution with strong boundary layer effects. Moreover, we estimate the convergence rate of the numerical solution to the unperturbed second-order elliptic system. Numerical results for both smooth solutions and the solutions with sharp layers confirm the theoretical prediction.
\end{abstract}

\date{\today}
\maketitle
\section{Introduction}
The strain gradient models have drawn great attention recently because they capture the size effect of nano-materials for plasticity~\cite{FleckHutchinson:1997} as well as for mechanical meta-materials~\cite{Khakalo:2020} by incorporating the higher-order strain gradient and the intrinsic material length scale into the constitutive relations. Studies from the perspective of modeling, mechanics and mathematics may date back to 1960s~\cites{Koiter:19641,Mindlin:1964,Hlavacek:19692}, while large-scale simulations are relatively recent~\cites{zervos:20092,Zybell:2012, Rud:2014, Phunpeng:2015}. Different methods such as H$^2$ conforming finite element methods~\cites{zervos:20091, Fisher:2010}, H$^1$ conforming mixed finite element methods~\cites{Aravas:2002,Phunpeng:2015}, nonconforming finite element methods~\cites{LiMingShi:2017,LiaoM:2019, LiMingWang:2021}, discontinuous Galerkin methods~\cite{Engel:2002}, isogeometric analysis~\cites{Klassen:2011,Niiranen:2016}, and meshless methods~\cite{Askes:2002} have been used to simulate the strain gradient elastic models with different complexity, just to mention a few. One of the numerical difficulties is that the number of the materials parameters is too large~\cite{Mindlin:1964}, another is that the materials parameters may cause boundary layer or numerical instability when they tend to certain critical values~\cite{Engel:2002}.

The strain gradient elasticity model proposed by Altan and Aifantis~\cite{Altan:1992} seems the simplest one among them because it contains only one material parameter besides the Lam\'e constants, while it still models the size effect adequately~\cite{Aifantis:1999}. This model is described by the following boundary value problem: 
\begin{equation}\label{eq:sgbvp}
\left\{
\begin{aligned}
(\iota^2\Delta-\boldsymbol I)\Lr{\mu\Delta \boldsymbol u+(\lambda+\mu)\na\divop\boldsymbol u}&=\boldsymbol f\quad&&\text{in\;}\Om,\\
\boldsymbol u=\pa_{\boldsymbol n}\boldsymbol u&=\boldsymbol 0\quad&&\text{on\;}\pa\Om,
\end{aligned}\right.
\end{equation}
where $\Omega\subset\mb{R}^2$ is a smooth domain, $\boldsymbol u:\Omega\to\mb{R}^2$ is the displacement, $\pa_{\boldsymbol n}\boldsymbol u$ is the normal derivative of $\boldsymbol u$, $\lam$ and $\mu$ are the Lam\'e constants, and $\iota$ is the microscopic parameter such that $0<\iota\le 1$, which stands for the intrinsic length scale. Besides modeling the strain gradient elasticity, the system~\eqref{eq:sgbvp} may also be regarded as a vector analog of the fourth-order singular perturbed problem, which usually models a clamped plate problem~\cites{John:1973,Schuss:1976, Semper:1992, Semper:1994, Tai:2001, Brenner:2011}, or arises from a fourth-order perturbation of the fully nonlinear Monge-Amp\`ere equation~\cites{Brenner:2011b, Feng:2014}. System of the form~\eqref{eq:sgbvp} may also come from the linearized model in MEMS~\cite{Laurencot:2017}.

In the present work, we are interested in~\eqref{eq:sgbvp} for the nearly incompressible materials. Such materials are commonly used in industry and a typical example is natural rubber. To the best of our knowledge, the studies on the approximation of incompressible and nearly incompressible strain gradient elasticity have not been sufficiently addressed in the literature, although vast efforts have been devoted to finite element approximation of the incompressible and nearly incompressible elasticity problems; See e.g.,~\cites{Hermann:1965,Vog:1983,Simo:1985,Babuska:1992, Brenner:1992, BraessMing:2005, Auricchio:2013}. In~\cite{FleckHutchinson:1997}*{\S III. C}, the authors studied the incompressible limit of the strain gradient model. Mixed finite elements for the incompressible Fleck-Hutchinson strain gradient model have been designed and tested in~\cite{ShuKingFleck:1999}. A finite element method has been tested for the nearly incompressible strain gradient model in~\cite{Wei:2006}. A mixed finite element, which approximated the displacement with Bogner-Fox-Schmidt element~\cite{BFS:1965} and approximated the pressure with the $9-$node quadrilateral element, was constructed for the coupled stress model in~\cite{Fisher:2011}, and bore certain similarities with problem~\eqref{eq:sgbvp}. Recently, Hu and Tian~\cite{Tian:2021} have proposed several robust elements for the two-dimensional strain gradient model in the framework of reduced integration. Unfortunately, none of the above work proved the robustness of the proposed elements rigorously in the incompressible limit. 

 Following the classical approach dealing with the nearly incompressible elasticity problem~\cite{Hermann:1965}, we introduce an auxiliary variable ``pressure'' $p$ and recast~\eqref{eq:sgbvp} into a displacement-pressure mixed variational problem, i.e., the so-called $(\boldsymbol u,p)-$formulation. We approximate the displacement by augmenting the finite element space in~\cite{GuzmanLeykekhmanNeilan:2012} with certain new bubble functions. The original motivation for the bubble functions is to design the stable finite element pair for the Stokes problem~\cites{Arnold:1984,Bernardi:1985}. The augmented bubble functions help out in dealing with the extra constraints such as the divergence stability in Stokes problem and the high order consistency error~\cites{Tai:2001,GuzmanLeykekhmanNeilan:2012,Zhang:2012}.  Such idea has been exploited by one of the authors to design robust finite elements for the strain gradient elasticity model~\cite{LiMingShi:2017}. Besides, we employ the standard continuous Lagrangian finite element of one order lower than that for the displacement to approximate the pressure. Such a finite element pair is similar to the Taylor-Hood element in the Stokes flows~\cite{Hood:1973} which is $\mb{P}_r-\mb{P}_{r-1}$ scheme and continuous pressure approximation. For both smooth solutions and solutions with strong boundary layer effects, these mixed finite element pairs are robust in the incompressible limit, here the robustness is understood in the sense that the rate of convergence is uniform in both $\iota$ and $\lambda$. The bubble function spaces in approximating the displacement are defined by certain orthogonal constraints, and the explicit representations of these spaces are desired for the sake of implementation. We achieve this with the aid of the Jacobi polynomial. In addition to perspicuous results in view of analytics, such representation lends itself to the construction of the analytical shape functions for the approximating space of the displacement. Though we focus on the two-dimensional problem, the element may be readily extended to the three-dimensional problem. cf., Remark~\ref{remark:3d}.

 By standard mixed finite element theory~\cite{BBF:2013}, a discrete B-B inequality that is uniform in $\iota$ is needed for the well-posedness of the mixed $(\boldsymbol u,p)-$discretization problem. This B-B inequality reduces to the remarkable B-B inequality for the Stokes problem when $\iota$ tends to zero. A natural way to prove the discrete B-B inequality is to construct a uniformly stable Fortin operator~\cites{Fortin:1977,Winther:2013,MardalTaiWinther:2002}, which does not seem easy due to the complication of the constraints. To this end, we construct a quasi-Fortin operator that takes different forms for small $\iota$ as well as large $\iota$. This quasi-Fortin operator is bounded in a weighted energy norm in the corresponding regimes of $\iota$. Besides the discrete B-B inequality, another ingredient in proving the robustness is a new regularity result for the solution of~\eqref{eq:sgbvp} that is uniform in both $\lam$ and $\iota$, which is crucial to prove the convergence rate for the layered solution. The proof combines the method dealing with the nearly incompressible linear elasticity~\cite{Vog:1983} and the regularity estimate for the fourth-order singular perturbed problem~\cites{Tai:2001, LiMingWang:2021}.

The outline of the paper is as follows. In \S 2, we introduce Altan and Aifantis' strain gradient model and its mixed variational formulation. We demonstrate the numerical difficulty caused by large $\lam$, and prove the uniform regularity estimate for problem~\eqref{eq:sgbvp}. In \S 3, we construct a family of nonconforming finite elements, and derive the explicit formulations for the bubble spaces. In \S 4, we use the nonconforming elements proposed in \S 3 together with the continuous Lagrangian finite elements to discretize the mixed variational problem and prove the optimal rate of convergence. In the last section, we report the numerical results, which are consistent with the theoretical prediction.

Throughout this paper, the constant $C$ may differ from line to line, while it is independent of the mesh size $h$, the materials parameter $\iota$ and the Lam\'e constant $\lambda$.
\section{The Mixed Variational Formulation and Regularity Estimates}
First we fix some notations. The space $L^2(\Om)$ of the square-integrable
functions defined on a smooth domain $\Om$ is equipped with the inner product $(\cdot,\cdot)$ and the norm $\nm{\cdot}{L^2(\Om)}$, while $L_0^2(\Omega)$ is the subspace of $L^2(\Omega)$ with mean value zero. Let $H^m(\Om)$ be the standard Sobolev space~\cite{AdamsFournier:2003} with the norm $\nm{\cdot}{H^m(\Om)}$, while $H_0^m(\Omega)$ is the closure in $H^m(\Omega)$ of $C_0^\infty(\Omega)$. We may drop $\Om$ in $\nm{\cdot}{H^m(\Om)}$  when no confusion may occur. For any vector-valued function $\boldsymbol v$, its gradient is a matrix-valued
function with components $(\na \boldsymbol v)_{ij}=\pa_jv_i$, and the symmetric part of $\na\boldsymbol v$ is defined by
\(
\boldsymbol\eps(\boldsymbol v)=(\na\boldsymbol v+[\na\boldsymbol v]^T)/2.
\)
The divergence operator is defined as $\divop\boldsymbol v=\pa_1v_1+\pa_2v_2$.
The Sobolev spaces $[H^m(\Om)]^2$, $[H_0^m(\Om)]^2$ and  $[L^2(\Om)]^2$ of a vector-valued function may be defined
similarly as their scalar counterpart. This rule equally applies to the inner products and the norms. The double inner product between tensors $\boldsymbol A=(A_{ij})_{i,j=1}^2,\boldsymbol B=(B_{ij})_{i,j=1}^2$ equals $\boldsymbol A:\boldsymbol B=\sum_{i,j=1}^2A_{ij}B_{ij}$.

We recast~\eqref{eq:sgbvp} into a variational problem: Find $\boldsymbol u\in V{:}=[H^2_0(\Om)]^2$ such that
\begin{equation}\label{eq:variation}
a(\boldsymbol u,\boldsymbol v)=(\boldsymbol f,\boldsymbol v)\quad\text{for all\quad} \boldsymbol v\in V,
\end{equation}
where
\(
a(\boldsymbol u,\boldsymbol v){:}=(\C\boldsymbol\eps(\boldsymbol u),\boldsymbol\eps(\boldsymbol v))+\iota^2(\D\na\boldsymbol\eps(\boldsymbol u),\na\boldsymbol\eps(\boldsymbol v)),
\)
and the fourth-order tensor $\C$ and the sixth-order tensor $\D$ are defined as
\[
\C_{ijkl}=\lam\del_{ij}\del_{kl}+2\mu\del_{ik}\del_{jl}\quad\text{and}\quad
\D_{ijklmn}=\lam\del_{il}\del_{jk}\del_{mn}+2\mu\del_{il}\del_{jm}\del_{kn},
\]
respectively. Here $\delta_{ij}$ is the Kronecker delta function. The strain gradient $\na\boldsymbol\eps(\boldsymbol v)$ is a third-order tensor that is defined by $(\na\boldsymbol\eps(\boldsymbol v))_{ijk}=\pa_i(\boldsymbol\eps(\boldsymbol v))_{jk}$.

We are interested in the regime when $\lambda\to\infty$, which means that the material is nearly incompressible. Proceeding along the same line that leads to~\cite{LiMingWang:2021}*{Theorem 5}, the tensor product of the  element (NTW) proposed in~\cite{Tai:2001} may be used to approximate~\eqref{eq:sgbvp}, and the error estimate reads as
\[
\|\boldsymbol u-\boldsymbol u_h\|\le C\lam(h^2+\iota h)\nm{\boldsymbol u}{H^3},
\]
where $\|\boldsymbol v\|^2{:}=a(\boldsymbol v,\boldsymbol v)$, and $C$ is independent of the mesh size $h$, and $\iota$ and $\lam$. Therefore, the error bound degenerates when $\lam$ is large, and the NTW element does not seem a good candidate for the nearly incompressible material. The following numerical example confirms this observation.
\begin{example}
Let $\Omega=(0,1)^2$, and $\boldsymbol u=(u_1,u_2)$ with
\[
u_1=-\sin^3(\pi x)\sin(2\pi y)\sin(\pi y),\quad u_2=\sin(2\pi x)\sin(\pi x)\sin^3(\pi y).
\]
It is clear that $\divop\boldsymbol u=0$, hence the material is completely incompressible. The details of the numerical experiment such as the mesh generation, are the same as those in \S~\ref{sec:numerical}. The relative error $\|\boldsymbol u-\boldsymbol u_h\|/\|\boldsymbol u\|$ in Table~\ref{tab:NTW} shows that the rate of convergence deteriorates when $\lam$ is large.
\begin{table}[htbp]\centering\caption{Relative errors and rate of convergence for NTW}~\label{tab:NTW}\begin{tabular}{ccccc}
\hline
$\iota\backslash h$ & 1/8 & 1/16 & 1/32 & 1/64\\
\hline
\multicolumn{5}{c}{$\nu=0.3000,\lambda=0.5769,\mu=0.3846$}\\
\hline
1e+00 & 2.681e-01 & 1.373e-01 & 6.698e-02 & 3.334e-02\\
rate & & 0.97 & 1.04 & 1.01\\
1e-06 & 4.550e-02 & 1.244e-02 & 3.001e-03 & 7.467e-04\\
rate & & 1.87 & 2.05 & 2.01\\
\hline
\multicolumn{5}{c}{$\nu=0.499999,\lambda=\text{1.6667e5}, \mu=0.3333$}\\
\hline
1e+00 & 9.995e-01 & 9.979e-01 & 9.916e-01 & 9.682e-01\\
rate & & 0.00 & 0.01 & 0.03\\
1e-06 & 9.752e-01 & 7.233e-01 & 2.502e-01 & 6.561e-02\\
rate & & 0.43 & 1.53 & 1.93\\
\hline
\end{tabular}\end{table}
\end{example}
\subsection{The mixed variational formulation}
We introduce an auxiliary variable
\(
p=\lam\divop\boldsymbol u,
\)
and $p\in P{:}=L_0^2(\Om)\cap H_0^1(\Om)$.  We write Problem~\eqref{eq:variation} into a mixed variational problem as
\begin{equation}\label{eq:mix}
\left\{
\begin{aligned}
a_{\iota}(\boldsymbol u,\boldsymbol v)+b_{\iota}(\boldsymbol v,p)&=(\boldsymbol f,\boldsymbol v)\qquad&&\text{for all\quad}\boldsymbol v\in V,\\
b_{\iota}(\boldsymbol u,q)-\lam^{-1}c_{\iota}(p,q)&=0\qquad&&\text{for all\quad}q\in P,
\end{aligned}\right.
\end{equation}
where
\[
\begin{aligned}
a_{\iota}(\boldsymbol v,\boldsymbol w){:}&=2\mu\lr{(\boldsymbol\eps(\boldsymbol v),\boldsymbol\eps(\boldsymbol w))+\iota^2(\na\boldsymbol\eps(\boldsymbol v),\na\boldsymbol\eps(\boldsymbol w))},\qquad &&\boldsymbol v,\boldsymbol w\in V,\\
b_{\iota}(\boldsymbol v,q){:}&=(\divop\boldsymbol v, q)+\iota^2(\na\divop\boldsymbol v,\na q),\qquad &&\boldsymbol v\in V, q\in P,\\
c_{\iota}(s,q){:}&=(s,q)+\iota^2(\na s,\na q),\qquad &&s,q\in P.
\end{aligned}
\]
It is convenient to define the weighted norm for all $q\in P$ as
\[
\inm{q}{:}=\nm{q}{L^2}+\iota\nm{\na q}{L^2}.
\]
$\inm{q}$ is a norm over $P$ for any $q\in P$ and any finite $\iota$. By Poincar\'e's inequality, $\inm{\na\boldsymbol v}$ is a norm over $V$ for any $\boldsymbol v\in V$. To study the well-posedness of Problem~\eqref{eq:mix}, we start with the following B-B inequality that is uniform for any $\iota$.
\begin{lemma}\label{lema:div}
For any $q\in P$, there exists $\boldsymbol v\in V$ such that
\begin{equation}\label{eq:estdiv1}
\divop\boldsymbol v=q\qquad\text{and}\qquad\inm{\na\boldsymbol v}\le C\inm{q},
\end{equation}
where $C$ only depends on $\Om$ but is independent of $\iota$.
\end{lemma}

\begin{proof}
By~\cite{Galdi:2011}*{Theorem III 3.3} and~\cite{Costabel:2010}*{Proposition 4.1}, for any $q\in P$, there exists $\boldsymbol v\in V$ such that $\divop\boldsymbol v=q$ and
\begin{equation}\label{eq:estdiv}
\nm{\boldsymbol v}{H^1}\le C\nm{q}{L^2}\qquad\text{and}\qquad\nm{\boldsymbol v}{H^2}\le C\nm{q}{H^1},
\end{equation}
where the constant $C$ only depends on $\Om$.

Because the mean of $q$ is zero for any $q\in P$, by Poincar\'e's inequality, there exists $C$ such that
\[
\nm{q}{H^1}\le C\nm{\na q}{L^2}.
\]

Combining the above two inequalities, we obtain
\[
\inm{\na\boldsymbol v}=\nm{\na\boldsymbol v}{L^2}+\iota\nm{\na^2\boldsymbol v}{L^2}\le C\inm{q}.
\]
This gives~\eqref{eq:estdiv1}.
\end{proof}
\begin{lemma}
There exists a unique $\boldsymbol u\in V$ and $p\in P$ satisfying~\eqref{eq:mix}, and there exists $C$ independent of $\iota$ and $\lam$ such that
\begin{equation}\label{eq:estsolu}
\inm{\na\boldsymbol u}+\inm{p}\le C\nm{\boldsymbol f}{H^{-1}}.
\end{equation}
\end{lemma}

\begin{proof}
By the first Korn's inequality~\cites{Korn:1908,Korn:1909},
\begin{equation}\label{eq:1stkorn}
\nm{\boldsymbol\eps(\boldsymbol v)}{L^2}^2\ge\dfrac12\nm{\na\boldsymbol v}{L^2}^2\qquad\text{for all\quad}\boldsymbol v\in [H_0^1(\Om)]^2,
\end{equation}
and the H$^2$ Korn's inequality~\cite{LiMingWang:2021}*{Theorem 1},
\begin{equation}\label{eq:h2korn}
\nm{\na\boldsymbol\eps(\boldsymbol v)}{L^2}^2\ge\Lr{1-\dfrac1{\sqrt2}}\nm{\na^2\boldsymbol v}{L^2}^2\qquad\text{for all\quad}\boldsymbol v\in [H^2(\Om)]^2,
\end{equation}
we obtain
\[
a_{\iota}(\boldsymbol v,\boldsymbol v)\ge 2\mu\Lr{\dfrac12\nm{\na\boldsymbol v}{L^2}^2+\lr{1-\dfrac1{\sqrt2}}\iota^2\nm{\na^2\boldsymbol v}{L^2}^2}\ge\dfrac{\mu}2\inm{\na\boldsymbol v}^2.
\]

Using~\eqref{eq:estdiv1}, for any $p\in P$, there exists $\boldsymbol v_0\in V$ such that
$\divop\boldsymbol v_0=p$ and $\inm{\na\boldsymbol v_0}\le C\inm{p}$. This implies
\[
\sup_{\boldsymbol v\in V}\dfrac{b_{\iota}(\boldsymbol v,p)}{\inm{\na\boldsymbol v}}\ge\dfrac{b_{\iota}(\boldsymbol v_0,p)}{\inm{\na\boldsymbol v_0}}
=\dfrac{\inm{p}^2}{\inm{\na\boldsymbol v_0}}\ge C\inm{p}.
\]

By~\cite{Braess:1996}*{Theorem 2}, we immediately obtain the well-posedness of~\eqref{eq:mix} and the estimate~\eqref{eq:estsolu} by noting 
\[
\abs{(\boldsymbol f,\boldsymbol v)}\le\nm{\boldsymbol f}{H^{-1}}\nm{\boldsymbol v}{H^1}\le C\nm{\boldsymbol f}{H^{-1}}\nm{\na\boldsymbol v}{L^2}\le C\nm{\boldsymbol f}{H^{-1}}\inm{\na\boldsymbol v}.
\]
\end{proof}

By the standard regularity theory for the elliptic system, we find $\boldsymbol u\in [H^4(\Omega)]^2$ and $p\in H^3(\Omega)$ provided that $\boldsymbol f\in [L^2(\Omega)]^2$, while we are interested in whether the shift estimate $\nm{\na^2\boldsymbol u}{\iota}+\nm{\na p}{\iota}\le  C(\iota)\nm{\boldsymbol f}{L^2}$ is true with a $\lam-$independent constant $C(\iota)$, this is the objective of the next part.
\subsection{Regularity estimates}
We aim to study the regularity of the solution of~\eqref{eq:sgbvp}.  Letting $\iota\to 0$, we find $\boldsymbol u_0\in[H_0^1(\Om)]^2$ satisfying
\begin{equation}\label{eq:elas}
-\mc{L}\boldsymbol u_0=\boldsymbol f\quad\text{in\;}\Om,\qquad \boldsymbol u_0=\boldsymbol 0\quad\text{on\;}\pa\Om,
\end{equation}
in the sense of distribution, where 
\(
\mc{L}\boldsymbol u_0{:}=\mu\Delta\boldsymbol u_0+(\lam+\mu)\na\divop\boldsymbol u_0
\).
The H$^1-$error for $\boldsymbol u-\boldsymbol u_0$ will be given in Theorem~\ref{thm:reg}, which is crucial for the regularity estimate of problem~\eqref{eq:sgbvp}. We reshape~\eqref{eq:elas} into a variational problem: Find $\boldsymbol u_0\in[H_0^1(\Om)]^2$ such that
\begin{equation}\label{eq:varaelas}
(\mb{C}\boldsymbol\eps(\boldsymbol u_0),\boldsymbol\eps(\boldsymbol v))=(\boldsymbol f,\boldsymbol v)\qquad\text{for all\;}\boldsymbol v\in [H_0^1(\Om)]^2.
\end{equation}

By~\cite{BacutaBramble:2003}, we have the following shift estimate for $\boldsymbol u_0$: There exists $C$ independent of $\lam$ such that
\begin{equation}\label{eq:regelas}
\nm{\boldsymbol u_0}{H^2}+\lam\nm{\divop\boldsymbol u_0}{H^1}\le C\nm{\boldsymbol f}{L^2}.
\end{equation}

Next we study an auxiliary boundary value problem:
\begin{equation}\label{eq:auxprob}
\left\{\begin{aligned}
\Delta\mc{L}\boldsymbol w&=\boldsymbol F,\qquad\text{in\quad}\Om,\\
\boldsymbol w=\pa_{\boldsymbol n}\boldsymbol w&=\boldsymbol 0,\qquad\text{on\quad}\pa\Om.
\end{aligned}\right.
\end{equation}
The a-priori estimate for the solution of the above boundary value problem reads as
\begin{lemma}
There exists a unique $\boldsymbol w\in V$ satisfying~\eqref{eq:auxprob}, and there exists $C$ independent of $\lam$ such that
\begin{equation}\label{eq:auxpriori}
\nm{\boldsymbol w}{H^2}+\lam\nm{\divop\boldsymbol w}{H^1}\le C\nm{\boldsymbol F}{H^{-2}}.
\end{equation}
\end{lemma}

\begin{proof}
We recast~\eqref{eq:auxprob} into a variational problem: Find $\boldsymbol w\in V$ such that
\[
A(\boldsymbol w,\boldsymbol v)=(\boldsymbol F,\boldsymbol v)\qquad\text{for all\quad}\boldsymbol v\in V,
\]
where $A(\boldsymbol v,\boldsymbol z){:}=2\mu(\na\boldsymbol\eps(\boldsymbol v),\na\boldsymbol\eps(\boldsymbol z))+\lam(\na\divop\boldsymbol v,\na\divop\boldsymbol z)$ for any $\boldsymbol v,\boldsymbol z\in V$.

For any $\boldsymbol v\in V$, by the H$^2-$Korn's inequality~\eqref{eq:h2korn} and Poincar\'e's inequality, there exists $C$ such that
\[
A(\boldsymbol v,\boldsymbol v)\ge 2\mu\nm{\na\boldsymbol\eps(\boldsymbol v)}{L^2}^2\ge\dfrac{\mu}2\nm{\na^2\boldsymbol v}{L^2}^2\ge C\nm{\boldsymbol v}{H^2}^2.
\]
The existence and uniqueness of $\boldsymbol w\in V$ follow from the Lax-Milgram theorem, and
\begin{equation}\label{eq:auxpriori1}
\nm{\na^2\boldsymbol w}{L^2}\le\nm{\boldsymbol w}{H^2}\le C\nm{\boldsymbol F}{H^{-2}}.
\end{equation}

Noting that $\divop\boldsymbol w\in P$, using~\eqref{eq:estdiv}, we obtain that, there exists $\boldsymbol v_0\in V$ such that 
\(
\divop\boldsymbol v_0=\divop\boldsymbol w,
\)
and 
\[
\nm{\na^2\boldsymbol v_0}{L^2}\le C\nm{\divop\boldsymbol w}{H^1}\le C\nm{\na\divop\boldsymbol w}{L^2}.
\]
A combination of the above two inequalities gives
\begin{align*}
\lam\nm{\na\divop\boldsymbol w}{L^2}^2&=\lam(\na\divop\boldsymbol w,\na\divop\boldsymbol v_0)=A(\boldsymbol w,\boldsymbol v_0)-2\mu(\na\boldsymbol\eps(\boldsymbol w),\na\boldsymbol\eps(\boldsymbol v_0))\\
&=(\boldsymbol F,\boldsymbol v_0)-2\mu(\na\boldsymbol\eps(\boldsymbol w),\na\boldsymbol\eps(\boldsymbol v_0))\\
&\le\nm{\boldsymbol F}{H^{-2}}\nm{\boldsymbol v_0}{H^2}+2\mu\nm{\na^2\boldsymbol w}{L^2}\nm{\na^2\boldsymbol v_0}{L^2}\\
&\le C\Lr{\nm{\boldsymbol F}{H^{-2}}+2\mu\nm{\na^2\boldsymbol w}{L^2}}\nm{\na^2\boldsymbol v_0}{L^2}\\
&\le C\nm{\boldsymbol F}{H^{-2}}\nm{\na\divop\boldsymbol w}{L^2}.
\end{align*}
This implies
\(
\lam\nm{\na\divop\boldsymbol w}{L^2}\le C\nm{\boldsymbol F}{H^{-2}},
\)
which together with~\eqref{eq:auxpriori1} and Poincar\'e's inequality gives~\eqref{eq:auxpriori}.
\end{proof}

Now we turn to prove the regularity estimate of problem~\eqref{eq:auxprob}. We consider an auxiliary elliptic system: For any $\wt{\boldsymbol F}\in[L^2(\Om)]^2$ and $\wt{G}\in H^1(\Om)$, find $\boldsymbol z\in V$ and $q\in P$ such that the following boundary value problem is valid in the sense of distribution:
\begin{equation}\label{eq:auxprob2}
\left\{\begin{aligned}
\mu\Delta^2\boldsymbol z+\na\Delta q&=\wt{\boldsymbol F}\qquad&&\text{in\quad}\Om,\\
\Delta\divop\boldsymbol z&=\wt{G}\qquad&&\text{in\quad}\Om,\\
\boldsymbol z=\pa_{\boldsymbol n}\boldsymbol z&=\boldsymbol 0\qquad&&\text{on\quad}\pa\Om,\\
q&=0\qquad&&\text{on\quad}\pa\Om.
\end{aligned}\right.
\end{equation}
\begin{lemma}
Let $\boldsymbol z\in[H^4(\Omega)]^2$ and $q\in H^3(\Omega)$ be the solution of~\eqref{eq:auxprob2}. Assume that $m$ is a nonnegative integer, then there exists $C$ depending only on $\Om$ and $\mu$ such that
\begin{equation}\label{eq:reg0}
\nm{\boldsymbol z}{H^{m+4}}+\nm{q}{H^{m+3}}\le C\Lr{\nm{\wt{\boldsymbol F}}{H^m}+\nm{\wt{G}}{H^{m+1}}+\nm{\boldsymbol z}{L^2}+\nm{q}{L^2}}.
\end{equation}
\end{lemma}

\begin{proof}
We write~\eqref{eq:auxprob2}$_1$ and~\eqref{eq:auxprob2}$_2$ as
\[
\begin{pmatrix}
\mu\Delta^2&0&\pa_x\Delta\\
0&\mu\Delta^2&\pa_y\Delta\\
\Delta\pa_x&\Delta\pa_y&0
\end{pmatrix}
\begin{pmatrix}
z_1\\
z_2\\
q
\end{pmatrix}=\begin{pmatrix}
\wt{F}_1\\
\wt{F}_2\\
\wt{G}
\end{pmatrix}.
\]
The symbol of the above system is
\[
\mc{L}(\xi)=\begin{pmatrix}
\mu\abs{\xi}^4&0&\xi_1\abs{\xi}^2\\
0&\mu\abs{\xi}^4&\xi_2\abs{\xi}^2\\
\xi_1\abs{\xi}^2&\xi_2\abs{\xi}^2&0
\end{pmatrix}.
\]
A direct calculation gives
\[
\abs{\det\mc{L}(\xi)}=\mu\abs{\xi}^{10}>0\qquad\text{if\quad}\xi\not=0.
\]
This means that the boundary value problem~\eqref{eq:auxprob2} is elliptic in the sense of Agmon-Douglis-Nirenberg~\cite{Agmon:1964}. Moreover, the boundary condition is pure Dirichlet, and it is straightforward to verify that the boundary condition satisfies the complementing condition~\cite{Agmon:1964}. The regularity estimate~\eqref{eq:reg0} follows from~\cite{Agmon:1964}.
\end{proof}

A direct consequence of the above lemma is the following regularity estimate for problem~\eqref{eq:auxprob}.
\begin{lemma}\label{lema:reg0}
Let $\boldsymbol w\in V$ be the solution of~\eqref{eq:auxprob}, there exists $C$ independent of $\lam$ such that
\begin{equation}\label{eq:reg2}
\nm{\boldsymbol w}{H^3}+\lam\nm{\divop\boldsymbol w}{H^2}\le C\nm{\boldsymbol F}{H^{-1}}.
\end{equation}
\end{lemma}

\begin{proof}
Using the standard elliptic regularity estimate, there exists a unique solution $\boldsymbol w\in [H^4(\Omega)]^2$ when $\boldsymbol F\in[L^2(\Omega)]^2$.  We introduce $\boldsymbol z=\boldsymbol w$ and $q=(\lam+\mu)\divop\boldsymbol w$, hence $\boldsymbol z\in[H^4(\Omega)]^2$ and $q\in H^3(\Omega)$ satisfy~\eqref{eq:auxprob2} with $\wt{\boldsymbol F}=\boldsymbol F$ and $\wt{G}=\Delta\divop\boldsymbol w$. 

By~\eqref{eq:reg0} with $m=0$, we obtain
\[
\nm{\boldsymbol w}{H^4}+(\lam+\mu)\nm{\divop\boldsymbol w}{H^3}\le C\Lr{\nm{\boldsymbol F}{L^2}+\nm{\divop\boldsymbol w}{H^3}+\nm{\boldsymbol w}{L^2}+(\lam+\mu)\nm{\divop\boldsymbol w}{L^2}}.
\]
Using the a-priori estimate~\eqref{eq:auxpriori}, we obtain
\[
\nm{\boldsymbol w}{H^4}+(\lam+\mu)\nm{\divop\boldsymbol w}{H^3}\le C_0\Lr{\nm{\boldsymbol F}{L^2}+\nm{\divop\boldsymbol w}{H^3}}.
\]
Now for $\lam+\mu>2C_0$, it follows from
\[
\nm{\boldsymbol w}{H^4}+(\lam+\mu)\nm{\divop\boldsymbol w}{H^3}\le C_0\nm{\boldsymbol F}{L^2}+\dfrac{\lam+\mu}{2}\nm{\divop\boldsymbol w}{H^3}
\]
that
\[
\nm{\boldsymbol w}{H^4}+\lam\nm{\divop\boldsymbol w}{H^3}\le 2C_0\nm{\boldsymbol F}{L^2}.
\]
Interpolating the above inequality with~\eqref{eq:auxpriori}, we obtain~\eqref{eq:reg2}.

If $\lam+\mu\le 2C_0$, then~\eqref{eq:reg2} follows from the standard regularity estimates~\cite{Agmon:1964} for problem~\eqref{eq:auxprob}.
\end{proof}

We turn to prove the regularity of problem~\eqref{eq:sgbvp} when $\boldsymbol f\in [L^2(\Om)]^2$. Let $\boldsymbol u$ and $\boldsymbol u_0$ be the solutions of~\eqref{eq:sgbvp} and~\eqref{eq:elas}, respectively. For any $\boldsymbol v\in [H_0^1(\Omega)\cap H^2(\Omega)]^2$, integration by parts gives
\begin{equation}\label{eq:ceq}
  (\mb{C}\boldsymbol\epsilon(\boldsymbol u),\boldsymbol\epsilon(\boldsymbol v))=-(\mc{L}\boldsymbol u,\boldsymbol v),
\end{equation}
and 
\begin{equation}\label{eq:deq}
\iota^2(\mb{D}\nabla\boldsymbol\epsilon(\boldsymbol u),\nabla\boldsymbol\epsilon(\boldsymbol v))
=\iota^2(\Delta\mc{L}\boldsymbol u,\boldsymbol v)+\iota^2\int_{\partial\Omega}{(\partial_{\boldsymbol n}\boldsymbol\sigma\boldsymbol n)}\cdot\partial_{\boldsymbol n}\boldsymbol v\mr{d}\sigma(\boldsymbol x),
\end{equation}
where $\boldsymbol\sigma{:}=2\mu\boldsymbol\epsilon(\boldsymbol u)+\lambda\divop\boldsymbol u\boldsymbol I$ and $(\partial_{\boldsymbol n}\boldsymbol\sigma)_{ij}{:}=\partial_{\boldsymbol n}\sigma_{ij}$. The boundary term in~\eqref{eq:deq} is derived by the fact $\pa_jv_i=n_j\pa_{\boldsymbol n}v_i+t_j\pa_{\boldsymbol t}v_i=n_j\partial_{\boldsymbol n}v_i$ and
\begin{equation}\label{eq:beq}
\begin{aligned}
  2\mu\int_{\partial\Omega}\partial_{\boldsymbol n}\boldsymbol\epsilon(\boldsymbol u):\boldsymbol\epsilon(\boldsymbol v)\mr{d}\sigma(\boldsymbol x)&+\lambda\int_{\partial\Omega}\partial_{\boldsymbol n}\divop\boldsymbol u\divop\boldsymbol v\mr{d}\sigma(\boldsymbol x)=\int_{\pa\Om}\pa_{\boldsymbol n}\boldsymbol\sigma:\nabla\boldsymbol v\mr{d}\sigma(\boldsymbol x)\\
  =&\int_{\pa\Om}\pa_{\boldsymbol n}\sigma_{ij}\pa_jv_i\mr{d}\sigma(\boldsymbol x)=\int_{\pa\Om}\pa_{\boldsymbol n}\sigma_{ij}n_j\pa_{\boldsymbol n}v_i\mr{d}\sigma(\boldsymbol x)=\int_{\pa\Om}(\partial_{\boldsymbol n}\boldsymbol\sigma\boldsymbol n)\cdot\pa_{\boldsymbol n}v\mr{d}\sigma(\boldsymbol x).
\end{aligned}\end{equation}
A combination of~\eqref{eq:ceq} and~\eqref{eq:deq} leads to 
\[
  (\mb{C}\boldsymbol\epsilon(\boldsymbol u),\boldsymbol\eps(\boldsymbol v))+\iota^2(\mb{D}\nabla\boldsymbol\epsilon(\boldsymbol u),\nabla\boldsymbol\epsilon(\boldsymbol v))=\iota^2\int_{\partial\Omega}(\partial_{\boldsymbol n}\boldsymbol\sigma\boldsymbol n)\cdot\partial_{\boldsymbol n}\boldsymbol v\mr{d}\sigma(\boldsymbol x)+(\boldsymbol f,\boldsymbol v),
\]
which together with~\eqref{eq:varaelas} yields
\begin{equation}\label{eq:aeq}
  (\mb{C}\boldsymbol\epsilon(\boldsymbol u-\boldsymbol u_0),\boldsymbol\eps(\boldsymbol v))+\iota^2(\mb{D}\nabla\boldsymbol\epsilon(\boldsymbol u),\nabla\boldsymbol\epsilon(\boldsymbol v))=\iota^2\int_{\partial\Omega}(\partial_{\boldsymbol n}\boldsymbol\sigma\boldsymbol n)\cdot\partial_{\boldsymbol n}\boldsymbol v\mr{d}\sigma(\boldsymbol x).  
\end{equation} 
This identity is the starting point of the proof. 

We shall frequently use the following multiplicative trace inequality. There exists $C$ that depends only on $\Om$ such that
\begin{equation}\label{eq:multrace}
\nm{\psi}{L^2(\pa\Om)}\le C\nm{\psi}{L^2(\Om)}^{1/2}\nm{\psi}{H^1(\Om)}^{1/2},\qquad\psi\in H^1(\Om).
\end{equation}

Using Lemma~\ref{lema:reg0}, we condense the regularity of problem~\eqref{eq:sgbvp} to estimating $\boldsymbol u-\boldsymbol u_0$ and $p-p_0$ in various norms, with $p=\lambda\divop\boldsymbol u$ and $p_0=\lambda\divop\boldsymbol u_0$.
%
\begin{lemma}
  There exists $C$ independent of $\iota$ and $\lambda$ such that
  \begin{equation}\label{eq:h3err}
    \nm{\boldsymbol u}{H^3}+\nm{p}{H^2}\le C\iota^{-2}(\nm{\epsilon(\boldsymbol u-\boldsymbol u_0)}{L^2}+\nm{p-p_0}{L^2}).
  \end{equation}
\end{lemma}

\begin{proof}
  We rewrite~\eqref{eq:sgbvp} as
  \[
    \left\{\begin{aligned}
    \Delta\mc{L}\boldsymbol u&=\iota^{-2}\mc{L}(\boldsymbol u-\boldsymbol u_0)\qquad&&\text{in\quad}\Omega,\\
    \boldsymbol u=\pa_{\boldsymbol n}\boldsymbol u&=\boldsymbol 0\qquad&&\text{on\quad}\pa\Omega.
    \end{aligned}\right.
  \]
Applying the regularity estimate~\eqref{eq:reg2} to the elliptic system~\eqref{eq:auxprob}, and using~\eqref{eq:ceq},
we obtain
  \[
    \nm{\boldsymbol u}{H^3}+\nm{p}{H^2}\le C\iota^{-2}\nm{\mc{L}(\boldsymbol u-\boldsymbol u_0)}{H^{-1}}
    \le C\iota^{-2}(\nm{\epsilon(\boldsymbol u-\boldsymbol u_0)}{L^2}+\nm{p-p_0}{L^2}).
  \]
\end{proof}

The next lemma is crucial to prove Theorem~\ref{thm:reg}, which transforms the estimate of $\nm{p-p_0}{\iota}$ in terms of 
$\nm{\na(\boldsymbol u-\boldsymbol u_0)}{\iota}$ besides a term concerning $\boldsymbol f$.
\begin{lemma}
  There exists $C$ independent of $\iota$ and $\lambda$ such that
  \begin{equation}\label{eq:diverr}
 \nm{p-p_0}{\iota}\le C(\nm{\na(\boldsymbol u-\boldsymbol u_0)}{\iota}+\iota^{1/2}\nm{\boldsymbol f}{L^2}).
  \end{equation}
\end{lemma}

\begin{proof}
By~\cite{Danchin:2013}*{Theorem 3.1} and Poincar\'e's inequality, there exists $\boldsymbol v_0\in[H^2(\Om)\cap H_0^1(\Om)]^2$ such that $\divop\boldsymbol v_0=\divop(\boldsymbol u-\boldsymbol u_0)$, and
  \begin{equation}\label{eq:estdiv2}
    \nm{\boldsymbol v_0}{H^1}\le C\nm{\divop(\boldsymbol u-\boldsymbol u_0)}{L^2},\qquad \nm{\boldsymbol v_0}{H^2}\le C\nm{\nabla\divop(\boldsymbol u-\boldsymbol u_0)}{L^2}.
  \end{equation}

Substituting $\boldsymbol v=\boldsymbol v_0$ into~\eqref{eq:aeq}, multiplying the resulting identity by $\lam$, we obtain
\begin{align}
 \nm{p-p_0}{\iota}^2&=-\iota^2(\na p_0,\na(p-p_0))-2\lam\mu\lr{(\boldsymbol\epsilon(\boldsymbol u-\boldsymbol u_0),\boldsymbol\epsilon(\boldsymbol v_0))+\iota^2(\nabla\boldsymbol\epsilon(\boldsymbol u),\nabla\boldsymbol\epsilon(\boldsymbol v_0))}\nn\\
    &\quad+\lam\iota^2\int_{\partial\Omega}(\partial_{\boldsymbol n}\boldsymbol\sigma\boldsymbol n)\cdot\partial_{\boldsymbol n}\boldsymbol v_0\mr{d}\sigma(\boldsymbol x).\label{eq:dividen}
 \end{align}

By the regularity estimate~\eqref{eq:regelas}, the first term may be bounded as
 \[
  \iota^2\abs{(\na p_0,\nabla(p-p_0))}\le\dfrac{\iota^2}{8}\nm{\na(p-p_0)}{L^2}^2+2\iota^2\nm{\nabla p_0}{L^2}^2\le\dfrac18\nm{p-p_0}{\iota}^2+C\iota^2\nm{\boldsymbol f}{L^2}^2.
 \]

Using the triangle inequality and~\eqref{eq:regelas} again, we obtain
\begin{equation}\label{eq:norm-equiv}
\iota\nm{\na\boldsymbol\eps(\boldsymbol u)}{L^2}\le\iota\nm{\na\boldsymbol\eps(\boldsymbol u-\boldsymbol u_0)}{L^2}+\iota\nm{\na\boldsymbol\eps(\boldsymbol u_0)}{L^2}
\le\nm{\na(\boldsymbol u-\boldsymbol u_0)}{\iota}+C\iota\nm{\boldsymbol f}{L^2}.
\end{equation}
Using~\eqref{eq:estdiv2} and the above inequality, we bound the second term as
\begin{align*}
2\lam\mu\abs{(\boldsymbol\epsilon(\boldsymbol u-\boldsymbol u_0),\boldsymbol\epsilon(\boldsymbol v_0))+\iota^2(\nabla\boldsymbol\epsilon(\boldsymbol u),\nabla\boldsymbol\epsilon(\boldsymbol v_0))}
&\le C(\nm{\boldsymbol\epsilon(\boldsymbol u-\boldsymbol u_0)}{L^2}+\iota\nm{\nabla\boldsymbol\epsilon(\boldsymbol u)}{L^2})\nm{p-p_0}{\iota}\\
&\le\dfrac18\nm{p-p_0}{\iota}^2+C\Lr{\nm{\na(\boldsymbol u-\boldsymbol u_0)}{\iota}^2+\iota^2\nm{\boldsymbol f}{L^2}^2}.
\end{align*}

Using~\eqref{eq:beq} and the definition of $\boldsymbol v_0$, we rewrite the boundary term as
 \begin{align*}
\lam\iota^2\int_{\pa\Omega}(\partial_{\boldsymbol n}\boldsymbol\sigma\boldsymbol n)\cdot\partial_{\boldsymbol n}\boldsymbol v_0\mr{d}\sigma(\boldsymbol x)
&=2\lam\mu\iota^2\int_{\pa\Omega}\partial_{\boldsymbol n}\boldsymbol\epsilon(\boldsymbol u):\boldsymbol\epsilon(\boldsymbol v_0)\mr{d}\sigma(\boldsymbol x)+\lambda^2\iota^2\int_{\pa\Omega}\pa_{\boldsymbol n}\divop\boldsymbol u\divop\boldsymbol v_0\mr{d}\sigma(\boldsymbol x)\\
&=2\lam\mu\iota^2\int_{\pa\Omega}\partial_{\boldsymbol n}\boldsymbol\epsilon(\boldsymbol u):\boldsymbol\epsilon(\boldsymbol v_0)\mr{d}\sigma(\boldsymbol x)
+\iota^2\int_{\pa\Omega}\pa_{\boldsymbol n}p(p-p_0)\mr{d}\sigma(\boldsymbol x)\\
&=2\lam\mu\iota^2\int_{\pa\Omega}\partial_{\boldsymbol n}\boldsymbol\epsilon(\boldsymbol u):\boldsymbol\epsilon(\boldsymbol v_0)\mr{d}\sigma(\boldsymbol x)-\iota^2\int_{\pa\Omega}\pa_{\boldsymbol n}p\; p_0\mr{d}\sigma(\boldsymbol x).
 \end{align*}

Recalling the trace inequality~\eqref{eq:multrace}, we estimate the first boundary term as
\begin{align*}
 2\lam\mu\iota^2\biggl|\int_{\partial\Omega}\partial_{\boldsymbol n}\boldsymbol\epsilon(\boldsymbol u):\boldsymbol\epsilon(\boldsymbol v_0)\mr{d}\sigma(\boldsymbol x)\biggr|
&\le 2\lam\mu\iota^2\nm{\pa_{\boldsymbol n}\boldsymbol\eps(\boldsymbol u)}{L^2(\pa\Om)}\nm{\boldsymbol\eps(\boldsymbol v_0)}{L^2(\pa\Om)}\\
&\le C\lam\iota^2\nm{\nabla\boldsymbol\epsilon(\boldsymbol u)}{L^2}^{1/2}\nm{\nabla\boldsymbol\epsilon(\boldsymbol u)}{H^1}^{1/2}\nm{\boldsymbol\epsilon(\boldsymbol v_0)}{L^2}^{1/2}\nm{\boldsymbol\epsilon(\boldsymbol v_0)}{H^1}^{1/2}.
\end{align*}
Using~\eqref{eq:estdiv2}, there exists $C$ independent of $\lam$ and $\iota$ such that
\[
\lam^2\nm{\boldsymbol\epsilon(\boldsymbol v_0)}{L^2}\nm{\boldsymbol\epsilon(\boldsymbol v_0)}{H^1}
\le C\nm{p-p_0}{L^2}\nm{\na(p-p_0)}{L^2}\le C\iota^{-1}\nm{p-p_0}{\iota}^2.
\]
Using~\eqref{eq:norm-equiv} to estimate $\nm{\na\boldsymbol\eps(\boldsymbol u)}{L^2}$ and using~\eqref{eq:h3err} to bound $\nm{\na\boldsymbol\eps(\boldsymbol u)}{H^1}$, we obtain
\begin{align*}
\nm{\nabla\boldsymbol\epsilon(\boldsymbol u)}{L^2}\nm{\nabla\boldsymbol\epsilon(\boldsymbol u)}{H^1}&\le C\iota^{-3}\Lr{\nm{\na(\boldsymbol u-\boldsymbol u_0)}{\iota}+\iota\nm{\boldsymbol f}{L^2}}
\Lr{\nm{\epsilon(\boldsymbol u-\boldsymbol u_0)}{L^2}+\nm{p-p_0}{L^2}}\\
&\le C\iota^{-3}
\lr{\nm{\na(\boldsymbol u-\boldsymbol u_0)}{\iota}^2+\iota^2\nm{\boldsymbol f}{L^2}^2+\Lr{\nm{\na(\boldsymbol u-\boldsymbol u_0)}{\iota}+\iota\nm{\boldsymbol f}{L^2}}\nm{p-p_0}{\iota}}.
\end{align*}
A combination of the above three inequalities gives
\begin{align*}
2\lam\mu\iota^2\biggl|\int_{\partial\Omega}\partial_{\boldsymbol n}\boldsymbol\epsilon(\boldsymbol u):\boldsymbol\epsilon(\boldsymbol v_0)\mr{d}\sigma(\boldsymbol x)\biggr|
\le&C\Bigl((\nm{\na(\boldsymbol u-\boldsymbol u_0)}{\iota}+\iota\nm{\boldsymbol f}{L^2})\nm{p-p_0}{\iota}\\
&\qquad+(\nm{\na(\boldsymbol u-\boldsymbol u_0)}{\iota}+\iota\nm{\boldsymbol f}{L^2})^{1/2}\nm{p-p_0}{\iota}^{3/2}\Bigr)\\
\le&\dfrac18\nm{p-p_0}{\iota}^2+C\Lr{\nm{\na(\boldsymbol u-\boldsymbol u_0)}{\iota}^2+\iota^2\nm{\boldsymbol f}{L^2}^2}. 
\end{align*}

Using the trace inequality~\eqref{eq:multrace} and the regularity estimate~\eqref{eq:regelas} again, we bound 
\[
\iota^2\biggl|\int_{\partial\Omega}\partial_{\boldsymbol n}pp_0\mr{d}\sigma(\boldsymbol x)\biggr|\le C\iota^2\nm{\na p}{L^2}^{1/2}\nm{\na p}{H^1}^{1/2}\nm{p_0}{H^1}\le C\iota^2\nm{\na p}{L^2}^{1/2}\nm{\na p}{H^1}^{1/2}\nm{\boldsymbol f}{L^2}.
\]
Using the triangle inequality and~\eqref{eq:regelas} again, we obtain
\[
  \iota\nm{\nabla p}{L^2}\le\nm{p-p_0}{\iota}+\iota\nm{\nabla p_0}{L^2}\le\nm{p-p_0}{\iota}+C\iota\nm{\boldsymbol f}{L^2},
\]
which together with~\eqref{eq:h3err} implies
\[
  \nm{\nabla p}{L^2}\nm{\na p}{H^1}\le C\iota^{-3}\lr{\nm{p-p_0}{\iota}^2+(\nm{\na(\boldsymbol u-\boldsymbol u_0)}{\iota}+\iota\nm{\boldsymbol f}{L^2})\nm{p-p_0}{\iota}+\iota\nm{\na(\boldsymbol u-\boldsymbol u_0)}{\iota}\nm{\boldsymbol f}{L^2}}.
\]
Combining the above three inequalities, we bound the second boundary term as
\begin{align*}
\iota^2\biggl|\int_{\partial\Omega}\partial_{\boldsymbol n}pp_0\mr{d}\sigma(\boldsymbol x)\biggr|\le&C\Bigl(\iota^{1/2}\nm{p-p_0}{\iota}\nm{\boldsymbol f}{L^2}+(\nm{\na(\boldsymbol u-\boldsymbol u_0)}{\iota}+\iota^{1/2}\nm{\boldsymbol f}{L^2})^{3/2}\nm{p-p_0}{\iota}^{1/2}\\
&\qquad+\iota\nm{\na(\boldsymbol u-\boldsymbol u_0)}{\iota}^{1/2}\nm{\boldsymbol f}{L^2}^{3/2}\Bigr)\\
&\le\dfrac18\nm{p-p_0}{\iota}^2+C\Lr{\nm{\na(\boldsymbol u-\boldsymbol u_0)}{\iota}^2+\iota\nm{\boldsymbol f}{L^2}^2}.
\end{align*}

Substituting the above inequalities into~\eqref{eq:dividen}, we obtain
\[
\nm{p-p_0}{\iota}^2\le\dfrac12\nm{p-p_0}{\iota}^2+C\Lr{\nm{\na(\boldsymbol u-\boldsymbol u_0)}{\iota}^2+\iota\nm{\boldsymbol f}{L^2}^2}.
\]
This immediately gives~\eqref{eq:diverr}.
\end{proof}

We are ready to prove the regularity of problem~\eqref{eq:sgbvp}.
\begin{theorem}\label{thm:reg}
There exists $C$ independent of $\iota$ and $\lambda$ such that
\begin{equation}\label{eq:h1err}
\nm{\boldsymbol u-\boldsymbol u_0}{H^1}+\nm{p-p_0}{L^2}\le C\iota^{1/2}\nm{\boldsymbol f}{L^2},
\end{equation}
and
\begin{equation}\label{eq:h2reg}
\nm{\boldsymbol u}{H^{2+k}}+\nm{p}{H^{1+k}}\le C\iota^{-1/2-k}\nm{\boldsymbol f}{L^2},\qquad k=0,1.
\end{equation}
\end{theorem}

The estimate~\eqref{eq:h1err} improves the known results~\cite{LiMingWang:2021}*{Lemma 1} in two aspects. It clarifies the
fact that the estimate is $\lam-$independent and it gives the convergence rate for the pressure. The rate $\iota^{1/2}$ is optimal even for the scalar counterpart; cf.,~\cite{Tai:2001}*{Lemma 5.1}.

\begin{proof}
Substituting $\boldsymbol v=\boldsymbol u-\boldsymbol u_0$ into~\eqref{eq:aeq}, we get
\[
a(\boldsymbol u-\boldsymbol u_0,\boldsymbol u-\boldsymbol u_0)=-\iota^2(\mb{D}\nabla\boldsymbol\epsilon(\boldsymbol u),\nabla\boldsymbol\epsilon(\boldsymbol u-\boldsymbol u_0))-\iota^2\int_{\partial\Omega}(\partial_{\boldsymbol n}\boldsymbol\sigma\boldsymbol n)\cdot\partial_{\boldsymbol n}\boldsymbol u_0\mr{d}\sigma(\boldsymbol x).
\]

Using the regularity estimate~\eqref{eq:regelas}, we bound the first term as
  \[
    \iota^2\abs{(\mb{D}\nabla\boldsymbol\epsilon(\boldsymbol u),\nabla\boldsymbol\epsilon(\boldsymbol u-\boldsymbol u_0))}\le\dfrac{\iota^2}4(\mb{D}\nabla\boldsymbol\epsilon(\boldsymbol u-\boldsymbol u_0),\nabla\boldsymbol\epsilon(\boldsymbol u-\boldsymbol u_0))+\iota^2(\mb{D}\nabla\boldsymbol\epsilon(\boldsymbol u_0),\nabla\boldsymbol\epsilon(\boldsymbol u_0))
    \le\dfrac{1}4a(\boldsymbol u-\boldsymbol u_0,\boldsymbol u-\boldsymbol u_0)+C\iota^2\nm{\boldsymbol f}{L^2}^2.
  \]

To bound the second term, we let $\boldsymbol v=\boldsymbol u_0$ in~\eqref{eq:beq} and obtain
\[
 \iota^2\int_{\partial\Omega}(\partial_{\boldsymbol n}\boldsymbol\sigma\boldsymbol n)\cdot\partial_{\boldsymbol n}\boldsymbol u_0\mr{d}\sigma(\boldsymbol x)=2\mu\iota^2\int_{\partial\Omega}\partial_{\boldsymbol n}\boldsymbol\epsilon(\boldsymbol u):\boldsymbol\epsilon(\boldsymbol u_0)\mr{d}\sigma(\boldsymbol x)+\iota^2\int_{\partial\Omega}\partial_{\boldsymbol n}\divop\boldsymbol u\; p_0\mr{d}\sigma(\boldsymbol x).
 \]
Invoking the trace inequality, using the fact $\nm{\na\divop\boldsymbol u}{L^2}\le\nm{\na\boldsymbol\eps(\boldsymbol u)}{L^2}$ and~\eqref{eq:regelas}, we obtain
 \begin{align*}
\iota^2\labs{\int_{\partial\Omega}(\partial_{\boldsymbol n}\boldsymbol\sigma\boldsymbol n)\cdot\partial_{\boldsymbol n}\boldsymbol u_0\mr{d}\sigma(\boldsymbol x)}
&\le C\iota^2\Lr{\nm{\nabla\boldsymbol\epsilon(\boldsymbol u)}{L^2}^{1/2}\nm{\nabla\boldsymbol\epsilon(\boldsymbol u)}{H^1}^{1/2}\nm{\boldsymbol\epsilon(\boldsymbol u_0)}{H^1}+\nm{\nabla\divop\boldsymbol u}{L^2}^{1/2}\nm{\nabla\divop\boldsymbol u}{H^1}^{1/2}\nm{p_0}{H^1}}\\
&\le C\iota^2\nm{\nabla\boldsymbol\epsilon(\boldsymbol u)}{L^2}^{1/2}\nm{\boldsymbol u}{H^3}^{1/2}\Lr{\nm{\boldsymbol u_0}{H^2}+\nm{p_0}{H^1}}\\
&\le C\iota^2\nm{\nabla\boldsymbol\epsilon(\boldsymbol u)}{L^2}^{1/2}\nm{\boldsymbol u}{H^3}^{1/2}\nm{\boldsymbol f}{L^2}.
\end{align*}
Substituting~\eqref{eq:diverr} into~\eqref{eq:h3err}, we obtain, 
\[
\nm{\boldsymbol u}{H^3}+\nm{p}{H^2}\le C\iota^{-2}(\nm{\na(\boldsymbol u-\boldsymbol u_0)}{\iota}+\iota^{1/2}\nm{\boldsymbol f}{L^2}).
\]
Invoking~\eqref{eq:norm-equiv} again, we get
\[
\iota^2\nm{\nabla\boldsymbol\epsilon(\boldsymbol u)}{L^2}^{1/2}\nm{\boldsymbol u}{H^3}^{1/2}\le C\iota^{1/2}(\nm{\na(\boldsymbol u-\boldsymbol u_0)}{\iota}+\iota^{1/2}\nm{\boldsymbol f}{L^2}).
\]
A combination of the above three inequalities gives
\[
 \iota^2\labs{\int_{\partial\Omega}(\partial_{\boldsymbol n}\boldsymbol\sigma\boldsymbol n)\cdot\partial_{\boldsymbol n}\boldsymbol u_0\mr{d}\sigma(\boldsymbol x)}
\le\dfrac14\nm{\na(\boldsymbol u-\boldsymbol u_0)}{\iota}^2+C\iota\nm{\boldsymbol f}{L^2}^2.
\]
 
Combining the above inequalities, we obtain
\[
 \nm{\na(\boldsymbol u-\boldsymbol u_0)}{\iota}\le C\iota^{1/2}\nm{\boldsymbol f}{L^2},
 \]
which together with~\eqref{eq:diverr} leads to
\[
\nm{p-p_0}{\iota}\le C\iota^{1/2}\nm{\boldsymbol f}{L^2}.
\]
The above two estimates immediately give~\eqref{eq:h1err} and
\[
\nm{\na^2(\boldsymbol u-\boldsymbol u_0)}{L^2}+\nm{\na(p-p_0)}{L^2}\le C\iota^{-1/2}\nm{\boldsymbol f}{L^2},
\]
which together with~\eqref{eq:regelas} gives~\eqref{eq:h2reg} with $k=0$.

Substituting~\eqref{eq:h1err} into~\eqref{eq:h3err}, we obtain the higher regularity estimate~\eqref{eq:h2reg} with $k=1$.
\end{proof}

A direct consequence of the above theorem is
\begin{coro}
  There exists $C$ independent of $\iota$ and $\lambda$ such that
  \begin{equation}\label{eq:h32reg}
    \nm{\boldsymbol u}{H^{3/2}}+\nm{p}{H^{1/2}}\le C\nm{\boldsymbol f}{L^2},
  \end{equation}
  and
  \begin{equation}\label{eq:h52reg}
    \nm{\boldsymbol u}{H^{5/2}}+\nm{p}{H^{3/2}}\le C\iota^{-1}\nm{\boldsymbol f}{L^2}.
  \end{equation}
\end{coro}

\begin{proof}
Using triangle inequality,~\eqref{eq:regelas} and~\eqref{eq:h2reg}, we obtain
\[
\nm{\boldsymbol u-\boldsymbol u_0}{H^2}+\nm{p-p_0}{H^1}\le\nm{\boldsymbol u}{H^2}+\nm{p}{H^1}+\nm{\boldsymbol u_0}{H^2}+\nm{p_0}{H^1}
    \le C\iota^{-1/2}\nm{\boldsymbol f}{L^2}.
\]
Interpolating the above inequality and~\eqref{eq:h1err}, we obtain
\[
    \nm{\boldsymbol u-\boldsymbol u_0}{H^{3/2}}+\nm{p-p_0}{H^{1/2}}\le C\nm{\boldsymbol f}{L^2}.
\]
Using~\eqref{eq:regelas}, we get 
\[
\nm{\boldsymbol u_0}{H^{3/2}}+\nm{p_0}{H^{1/2}}\le C\nm{\boldsymbol u_0}{H^2}+\nm{p_0}{H^1}\le C\nm{\boldsymbol f}{L^2}.
\]
A combination of the above two inequalities and the triangle inequality leads to~\eqref{eq:h32reg}.

Interpolating~\eqref{eq:h2reg} with $k=0$ and~\eqref{eq:h2reg} with $k=1$, we obtain~\eqref{eq:h52reg}.
\end{proof}
\section{A Family of Nonconforming Finite Elements}
We introduce a family of finite elements to approximate the mixed variational problem~\eqref{eq:mix}. Let $\T_h$ be a triangulation of $\Om$ with maximum mesh size $h$. We assume all elements in $\T_h$ are shape-regular in the sense of Ciarlet and Raviart~\cite{Ciarlet:1978}. We also assume that $\mc{T}_h$ satisfies the inverse assumption: there exists $\sigma_0$ such that $h/h_K\le\sigma_0$ for all $K\in\mc{T}_h$. The space of piecewise vector fields is defined by
\[
[H^m(\Omega,\mc{T}_h)]^2{:}=\set{\boldsymbol v\in [L^2(\Omega)]^2}{\boldsymbol v|_K\in[H^m(T)]^2\quad\text{for all\quad}K\in\mc{T}_h},
\]
which is equipped with the broken norm
\[
\nm{\boldsymbol v}{H^m(\Omega,\mc{T}_h)}{:}=\nm{\boldsymbol v}{L^2}+\sum_{k=1}^m\nm{\na^k_h\boldsymbol v}{L^2},
\]
where
\(
\nm{\na^k_h\boldsymbol v}{L^2}^2=\sum_{K\in\mc{T}_h}\nm{\na^k\boldsymbol v}{L^2(K)}^2.
\)
For an interior edge $e$ shared by the triangles $K^+$ and $K^-$, we define the jump of $\boldsymbol v$ across $e$ as
\[
\jump{\boldsymbol v}{:}=\boldsymbol v^+\boldsymbol n^++\boldsymbol v^-\boldsymbol n^-\qquad\text{with}\quad \boldsymbol v^\pm=v|_{K^\pm},
\]
where $\boldsymbol n^\pm$ is the unit normal vector of $e$ towards the outside of $K^\pm$. For $e\cap\pa\Om\not=\emptyset$, we set $\jump{\boldsymbol v}=\boldsymbol v\boldsymbol n$.
\subsection{A family of finite elements}
Our construction is motivated by the element proposed in~\cite{GuzmanLeykekhmanNeilan:2012}. Define the element with a triple $(K,P_K,\Sigma_K)$ by specifying $K$ as a triangle, and 
\begin{equation}\label{eq:PK}
P_K{:}=\mathbb{P}_r(K)+b_K\sum_{i=1}^3b_iQ_i^{r-2}(K)+b_K^2R^{r-2}(K),
\end{equation}
where $b_K=\prod_{i=1}^3\lambda_i$ and $b_i=b_K/\lambda_i$ with $\{\lambda_i\}_{i=1}^3$ the barycentric coordinates of $K$. 

Define
\begin{equation}\label{eq:bubble1}
Q_i^{r-2}(K){:}=\set{v\in\mathbb{P}_{r-2}(K)}{\int_K b_Kb_ivq\dx=0\textup{ for all }q\in\mathbb{P}_{r-3}(K)},
\end{equation}
and
\begin{equation}\label{eq:bubble2}
R^{r-2}(K){:}=\set{v\in\mathbb{P}_{r-2}(K)}{\int_K b_K^2vq\dx=0\textup{ for all }q\in\mathbb{P}_{r-3}(K)}.
\end{equation}
The degrees of freedom (DoFs) for $P_K$ are given by
\begin{equation}\label{eq:dom}
v\mapsto\left\{\begin{aligned}
v(a)&\quad\text{for all vertices}\,a,\\
\int_evq\dsx&\quad\text{for all edges $e$ and\,}q\in\mb{P}_{r-2}(e),\\
\int_e\pa_{\boldsymbol n}vq\dsx&\quad\text{for all edges $e$ and\,}q\in\mb{P}_{r-2}(e),\\
\int_Kvq\dx&\quad\text{for all\,}q\in\mb{P}_{r-2}(K).
\end{aligned}\right.
\end{equation}
We plot the DoFs for $r=2,3$ in Figure~\ref{fig:Diagram}.
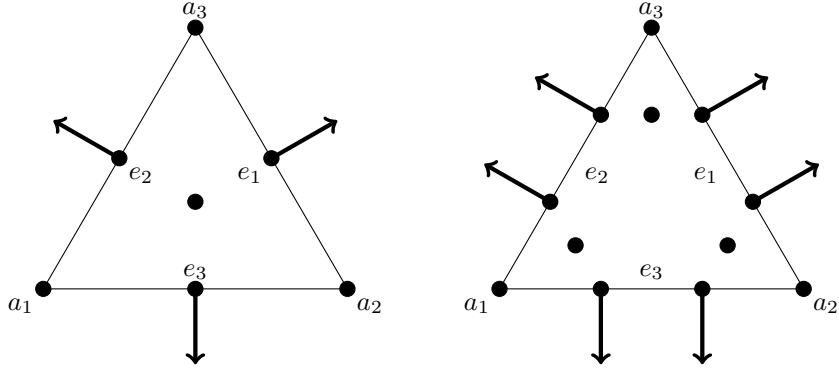
\begin{figure}[htbp]\centering\begin{tikzpicture}
\draw(-2,0)--(2,0)--(0,3.464)--(-2,0);
\draw[fill](-2,0)circle[radius=0.1];
\draw[fill](2,0)circle[radius=0.1];
\draw[fill](0,3.464)circle[radius=0.1];
\draw[fill](1,1.732)circle[radius=0.1];
\draw[fill](-1,1.732)circle[radius=0.1];
\draw[fill](0,0)circle[radius=0.1];
\draw[->][ultra thick](1,1.732)--(1.866,2.232);
\draw[->][ultra thick](-1,1.732)--(-1.866,2.232);
\draw[->][ultra thick](0,0)--(0,-1);
\draw[fill](0,1.155)circle[radius=0.1];
\node[below left]at(-2,0){$a_1$};
\node[below right]at(2,0){$a_2$};
\node[above]at(0,3.47){$a_3$};
\node[below left]at(1,1.732){$e_1$};
\node[below right]at(-1,1.732){$e_2$};
\node[above]at(0,0){$e_3$};
\draw(4,0)--(8,0)--(6,3.464)--(4,0);
\draw[fill](4,0)circle[radius=0.1];
\draw[fill](8,0)circle[radius=0.1];
\draw[fill](6,3.464)circle[radius=0.1];
\draw[fill](7.333,1.155)circle[radius=0.1];
\draw[fill](6.667,2.309)circle[radius=0.1];
\draw[fill](5.333,2.309)circle[radius=0.1];
\draw[fill](4.667,1.155)circle[radius=0.1];
\draw[fill](5.333,0)circle[radius=0.1];
\draw[fill](6.667,0)circle[radius=0.1];
\draw[->][ultra thick](7.333,1.155)--(8.2,1.655);
\draw[->][ultra thick](6.667,2.309)--(7.533,2.809);
\draw[->][ultra thick](5.333,2.309)--(4.467,2.809);
\draw[->][ultra thick](4.667,1.155)--(3.8,1.655);
\draw[->][ultra thick](5.333,0)--(5.333,-1);
\draw[->][ultra thick](6.667,0)--(6.667,-1);
\draw[fill](5,0.577)circle[radius=0.1];
\draw[fill](7,0.577)circle[radius=0.1];
\draw[fill](6,2.309)circle[radius=0.1];
\node[below left]at(4,0){$a_1$};
\node[below right]at(8,0){$a_2$};
\node[above]at(6,3.47){$a_3$};
\node[below left]at(7,1.732){$e_1$};
\node[below right]at(5,1.732){$e_2$};
\node[above]at(6,0){$e_3$};
\end{tikzpicture}\caption{Diagram for DoFs. Left: DoFs for $r=2$ are point evaluations of the function values at the vertex, the mean of the function along each edge, the mean of the normal derivative along each edge, and the mean of the function over the element; Right: DoFs for $r=3$  are point evaluations of the function values at the vertex, the means of the function against $\mb{P}_1$ along each edge, the means of the normal derivative against $\mb{P}_1$ along each edge, and the means of the function against $\mb{P}_1$ over the element}\label{fig:Diagram}
\end{figure}

\begin{lemma}\label{lema:unisolvent}
The set $(K,P_K,\Sigma_K)$ is unisolvent.
\end{lemma}

\begin{proof}
Firstly we show that if the DoFs~\eqref{eq:dom} can determine an element in $P_K$, then the element is unique. Suppose $v\in P_K$ vanishes at the DoFs listed in~\eqref{eq:dom}, it suffices to show $v\equiv 0$.
Assume that
\[
v=p_r+b_K\sum_{i=1}^3b_iq_i+b_K^2q_r,
\]
where $p_r\in\mb{P}_r(K)$, and $q_i\in Q_i^{r-2}(K)$, and $q_r\in R^{r-2}(K)$. DoFs associated with 
$\mb{P}_r(K)$ are
\[
v\mapsto\left\{\begin{aligned}
v(a)&\quad\text{for all vertices}\,a,\\
\int_evq\dsx&\quad\text{for all edges $e$ and\,}q\in\mb{P}_{r-2}(e), \\
\int_Kvq\dx&\quad\text{for all\,}q\in\mb{P}_{r-3}(K).
\end{aligned}\right.
\]
The bubble space vanishes on this subset of DoFs by~\eqref{eq:bubble1} and~\eqref{eq:bubble2}. The number of the DoFs is $3+3(r-1)+(r-1)(r-2)/2=(r+1)(r+2)/2$, which equals to the cardinality of $\mb{P}_r(K)$. Hence $p_r\equiv 0$. 

A direct calculation gives
\[
\int_{e_i}\pa_{\boldsymbol n}vq\dsx=-\abs{\na\lam_i}\int_{e_i}b_i^2q_iq\dsx=0\qquad\text{for all }q\in\mb{P}_{r-2}(e_i).
\]
Taking $q=q_i$ in the above identity, we obtain $q_i=0$ on $e_i$. Therefore, we write $q_i=\lam_ip_{r-3}$ for certain $p_{r-3}\in\mb{P}_{r-3}(K)$. Using~\eqref{eq:bubble1}, we get
\[
\int_Kb_Kb_iq_iq\dx=\int_Kb_K^2p_{r-3}q\dx=0\qquad\text{for all }q\in\mb{P}_{r-3}(K).
\]
Taking $q=p_{r-3}$ in the above identity, we obtain $p_{r-3}\equiv 0$. Therefore $v=b_K^2q_r$ for certain $q_r\in R^{r-2}(K)$. The last set of DoFs equals zero, i.e.,
\[
\int_Kb_K^2q_rq\dx=0\qquad\text{for all }q\in\mb{P}_{r-2}(K).
\]
Taking $q=q_r$ in the above identity, we obtain $q_r\equiv 0$. So does $v$. 

It remains to show the dimension of $P_K$ equals the number of DoFs~\eqref{eq:dom}. Proceeding along the same line as above, the element $v\equiv 0$ has a unique representation. Therefore~\eqref{eq:PK} is a direct sum, and
\[
  \text{dim}P_K=\text{dim}\mb{P}_r(K)+4\Lr{\text{dim}\mb{P}_{r-2}(K)-\text{dim}\mb{P}_{r-3}(K)}=\dfrac12(r^2+11r-6),
\]
which equals to the number of DoFs~\eqref{eq:dom} exactly.
\end{proof}
We define a local interpolation operator $\pi_K:H^2(K)\mapsto P_K$ as:
\begin{equation}\label{eq:inter1}
\left\{\begin{aligned}
\pi_K v(a)&=v(a)&&\quad\text{for all vertices\,}a,\\
\int_e\pi_Kvq\dsx&=\int_evq\dsx&&\quad\text{for all edges $e$ and\,}q\in\mb{P}_{r-2}(e),\\
\int_e\pa_{\boldsymbol n}\pi_Kvq\dsx&=\int_e\pa_{\boldsymbol n}vq\dsx&&\quad\text{for all edges $e$ and\,}q\in\mb{P}_{r-2}(e),\\
\int_K\pi_K vq\dx&=\int_Kvq\dx&&\quad\text{for all\,}q\in\mb{P}_{r-2}(K).
\end{aligned}\right.
\end{equation}
\begin{lemma}
There exists $C$ independent of $h_K$ such that for $v\in H^k(K)$ with $2\le k\le r+1$, there holds
\begin{equation}\label{eq:localerr}
\nm{\na^j(v-\pi_Kv)}{L^2(K)}\le Ch_K^{k-j}\nm{\na^kv}{L^2(K)},
\end{equation}
where $0\le j\le k$.
\end{lemma}

\begin{proof}
For any $v\in\mb{P}_r(K)\subset P_K$, the definition~\eqref{eq:inter1} shows that $v-\pi_Kv\in P_K$ and all DoFs of $v-P_K v$ vanish, then $v=\pi_Kv$. The estimate~\eqref{eq:localerr} immediately follows from the $\mb{P}_r(K)-$invariance of the local interpolation operator $\pi_K$~\cite{CiarletRaviart:1972}.
\end{proof}

\begin{remark}\label{remark:3d}
The element has a natural extension to three-dimensions by specifying $K$ as a tetrahedron, and
\[
P_K{:}=\mathbb{P}_r(K)+ b_K\sum_{i=1}^4b_iQ_i^{r-2}(K)+ b_K^2R^{r-2}(K),
\]
where $b_K=\prod_{i=1}^4\lam_i$ is the element bubble function with $\lambda_i$ the barycentric coordinates associated with the vertices $a_i$ for $i=1,\cdots,4$. $b_i=b_K/\lambda_i$ is the face bubble function associated with the face $f_i$. 

Define
\[Q_i^{r-2}(K){:}=\set{v\in\mathbb{P}_{r-2}(K)}{\int_K b_Kb_ivq\dx=0\textup{ for all }q\in\mathbb{P}_{r-3}(K)},\]
and
\[R^{r-2}(K){:}=\set{v\in\mathbb{P}_{r-2}(K)}{\int_K b_K^2vq\dx=0\textup{ for all }q\in\mathbb{P}_{r-4}(K)}.\]
The DoFs for $P_K$ are given by
\[
v\mapsto\left\{\begin{aligned}
v(a)&\quad\text{for all vertices},\\
\int_evq\dsx&\quad\text{for all edges $e$ and\,}q\in\mb{P}_{r-2}(e),\\
\int_fvq\dsx&\quad\text{for all faces $f$ and\,}q\in\mb{P}_{r-3}(f),\\
\int_f\pa_{\boldsymbol n}vq\dsx&\quad\text{for all faces $f$ and\,}q\in\mb{P}_{r-2}(f),\\
\int_Kvq\dx&\quad\text{for all\,}q\in\mb{P}_{r-2}(K).
\end{aligned}\right.
\]
Similar to Lemma~\ref{lema:unisolvent}, the set $(K,P_K,\Sigma_K)$ is also unisolvent.
\end{remark}
\subsection{Explicit representation for the bubble space}
We clarify the structures of~\eqref{eq:bubble1} and~\eqref{eq:bubble2} associated with the set of DoFs~\eqref{eq:dom}$_3$ and the subset of~\eqref{eq:dom}$_4$ respectively, and derive the explicit formulations of the corresponding shape functions, which seems missing in the literature, while such explicit representations are useful for implementation. We firstly recall the following facts about the Jacobi polynomials~\cite{Szego:1975}. For any $\alpha,\beta>-1$ and nonnegative integers $n,m$, there holds
\begin{equation}\label{eq:orthJacobi1}
\int_{-1}^1(1-t)^{\alpha}(1+t)^{\beta}P_n^{(\alpha,\beta)}(t)P_m^{(\alpha,\beta)}(t)\md t=h_n^{(\alpha,\beta)}\delta_{nm},
\end{equation}
where
\[
h_n^{(\alpha,\beta)}=\int_{-1}^1(1-t)^{\alpha}(1+t)^{\beta}\bigl[P_n^{(\alpha,\beta)}(t)\bigr]^2\md t.
\]
By~\cite{Szego:1975}*{Eq. (4.3.3)}, we may write
\begin{equation}\label{eq:squaremoment1}
h_n^{(\alpha,\beta)}=\dfrac{2^{\alpha+\beta+1}}{2n+\alpha+\beta+1}\dfrac{\Gamma(n+\alpha+1)\Gamma(n+\beta+1)}
{\Gamma(n+\alpha+\beta+1)\Gamma(n+1)},
\end{equation}
where $\Gamma$ is the Gamma function.

One of the explicit form for $P_n^{(\alpha,\beta)}$ is
\[(1- t)^{\alpha}(1+ t)^{\beta}P_n^{(\alpha,\beta)}(t)=\dfrac{(-1)^n}{2^n n!}\dfrac{\mathrm{d}^n}{\md t^n}
\Lr{(1- t)^{n+\alpha}(1+ t)^{n+\beta}}.\]
In particular,
\[
P_0^{(\alpha,\beta)}(t)=1,\quad P_1^{(\alpha,\beta)}( t)=\dfrac12(\alpha+\beta+2) t+\dfrac12(\alpha-\beta).
\]

Next we list certain facts about the Jacobi polynomials on the triangle~\cite{Dunkl:2014}*{Section 2.4}. For a triangle $K$ with vertices $a_1,a_2,a_3$, any point $x\in K$ is uniquely expressed as
\[x=\lam_1a_1+\lam_2a_2+\lam_3a_3, \quad \lam_i\ge 0 \text{ and }\lam_1+\lam_2+\lam_3=1.\]
Then $(\lam_1,\lam_2,\lam_3)$ is the barycentric coordinates of the point $x$ with respect to $K$. For nonnegative integers $k,n$ such that $k\le n$, we define
\begin{equation}\label{eq:Jacobi2}
P_{k,n}^{(\alpha,\beta,\gamma)}(\lam_1,\lam_2,\lam_3){:}=(\lam_2+\lam_3)^kP_{n-k}^{(2k+\beta+\gamma+1,\alpha)}(\lam_1-\lam_2-\lam_3)P_k^{(\gamma,\beta)}((\lam_2-\lam_3)/(\lam_2+\lam_3)).
\end{equation}
It is straightforward to verify $P_{k,n}^{(\alpha,\beta,\gamma)}\in \mb{P}_n(K)$. In particular,
\[
P_{0,0}^{(\alpha,\beta,\gamma)}(\lam_1,\lam_2,\lam_3)=1,\quad P_{0,1}^{(\alpha,\beta,\gamma)}(\lam_1,\lam_2,\lam_3)=(\beta+\gamma+2)\lam_1-(\alpha+1)(\lam_2+\lam_3),
\]
and
\[P_{1,1}^{(\alpha,\beta,\gamma)}(\lam_1,\lam_2,\lam_3)=(\gamma+1)\lam_2-(\beta+1)\lam_3.\]
For all $\alpha,\beta,\gamma>-1$, and nonnegative integers $j,k,m,n$ such that $j\le m$ and $k\le n$, there holds
\begin{equation}\label{eq:orthJacobi2}\begin{aligned}
&\negint_K\lam_1^\alpha \lam_2^\beta\lam_3^\gamma P_{k,n}^{(\alpha,\beta,\gamma)}(\lam_1,\lam_2,\lam_3)P_{j,m}^{(\alpha,\beta,\gamma)}(\lam_1,\lam_2,\lam_3)\dx\\
=&2\int_{\wh{K}}\lam_1^\alpha\lam_2^\beta(1-\lam_1-\lam_2)^\gamma P_{k,n}^{(\alpha,\beta,\gamma)}(\lam_1,\lam_2,1-\lam_1-\lam_2)P_{j,m}^{(\alpha,\beta,\gamma)}(\lam_1,\lam_2,1-\lam_1-\lam_2)\md\lam_1\md\lam_2\\
=&2h_{k,n}^{(\alpha,\beta,\gamma)}\delta_{jk}\delta_{mn},
\end{aligned}\end{equation}
where $\wh{K}:=\set{(\lam_1,\lam_2)}{\lam_1\ge 0,\lam_2\ge 0,\lam_1+\lam_2\le 1}$ is the standard reference triangle and
\begin{equation}\label{eq:squaremoment2}\begin{aligned}
h_{k,n}^{(\alpha,\beta,\gamma)}=&2^{-(2k+\alpha+2\beta+2\gamma+3)}h_{n-k}^{(2k+\beta+\gamma+1,\alpha)}h_k^{(\gamma,\beta)}\\
=&\dfrac1{(2n+\alpha+\beta+\gamma+2)(2k+\beta+\gamma+1)}\\
&\times\dfrac{\Gamma(n-k+\alpha+1)\Gamma(n+k+\beta+\gamma+2)\Gamma(k+\beta+1)\Gamma(k+\gamma+1)}{\Gamma(n-k+1)\Gamma(n+k+\alpha+\beta+\gamma+2)\Gamma(k+1)\Gamma(k+\beta+\gamma+1)}.
\end{aligned}\end{equation}
Using the notation $(x)_n=\Gamma(x+n)/\Gamma(n)$, we may find that the expression~\eqref{eq:squaremoment2} is equivalent to the one in~\cite{Dunkl:2014}*{Eq. (2.4.3)}. The identity~\eqref{eq:orthJacobi2} illustrates that $\{P_{k,n}^{(\alpha,\beta,\gamma)}(\lam_1,\lam_2,\lam_3)\mid 0\le k\le n\le r\}$ are mutually orthogonal bases of $\mb{P}_r(K)$ with respect to the weight $\lam_1^\alpha\lam_2^\beta\lam_3^\gamma$. 

Next we study the structure of the bubble spaces.  For the barycentric coordinate function $\lam_i$ such that $\lambda_i\equiv0$ on $e_i$, let $\lambda_i^+$ and $\lambda_i^-$  be the two other barycentric coordinates associated with the edges $e_i^+$ and $e_i^-$, respectively. $(e_i,e_i^+,e_i^-)$ are chosen in a counterclockwise manner. The space $Q_i^{r-2}(K)$ can be clarified by the Jacobi polynomials with respect to the weight $b_Kb_i$, while $R^{r-2}(K)$ can be clarified by the Jacobi polynomials with respect to the weight $b_K^2$, which are formulated in the following lemmas.
%
\begin{lemma}
The space $Q_i^{r-2}(K)$ takes the form
\[
Q_i^{r-2}(K)=\textup{span}\set{P_{k,r-2}^{(1,2,2)}(\lambda_i,\lambda_i^+,\lambda_i^-)}{0\le k\le r-2}.
\]
\end{lemma}

\begin{proof}
For any $v\in Q_i^{r-2}(K)\subset\mathbb{P}_{r-2}(K)$, $v$ may be expanded into 
\[
v=\sum_{0\le k\le n\le r-2}a_{kn}P_{k,n}^{(1,2,2)}(\lambda_i,\lambda_i^+,\lambda_i^-)
\]
with unknown parameters $a_{kn}$. Using the above representation, we may write the constraint in the definition~\eqref{eq:bubble1} as
\[
\sum_{0\le k\le n\le r-2}a_{kn}\negint_Kb_Kb_iP_{k,n}^{(1,2,2)}(\lambda_i,\lambda_i^+,\lambda_i^-)q\mathrm{d}x=0\qquad\text{for all\quad}q\in\mathbb{P}_{r-3}(K).
\]
Substituting $q=P_{j,m}^{(1,2,2)}(\lambda_i,\lambda_i^+,\lambda_i^-)$ for $0\le j\le m\le r-3$ into the above equation, and using the orthogonal relation~\eqref{eq:orthJacobi2}, we obtain $a_{jm}=0$ for $0\le j\le m\le r-3$. This concludes the lemma.
\end{proof}

Motivated by the above lemma, we change the definition of DoFs for the bubble space $b_K\sum_{i=1}^3b_iQ_i^{r-2}(K)$ from $\int_{e_i}\pa_{\boldsymbol n}vq\dsx$ for any $q\in\mb{P}_{r-2}(e_i)$ to
\[
\negint_{e_i}\diff{v}{n}P_k^{(2,2)}(\lambda_i^+-\lambda^-)\dsx, \quad k=0,\cdots, r-2.
\]

\begin{lemma}
The shape functions for the bubble space $b_K\sum_{i=1}^3b_iQ_i^{r-2}(K)$ associated with the above definition of DoFs are
\[
a_{k,r-2}b_Kb_iP_{k,r-2}^{(1,2,2)}(\lambda_i,\lambda_i^+,\lambda_i^-)
\]
with 
\begin{equation}\label{eq:ak}
a_{k,r-2}=\dfrac{(-1)^{r-k-1}}{\abs{\na\lam_i}}\dfrac{(k+3)(k+4)(2k+5)}{(r-k-1)(k+1)(k+2)}\qquad k=0,\cdots,r-2.
\end{equation}
\end{lemma}

\begin{proof}
A direct calculation gives
\[
\diff{}{\boldsymbol n}\Lr{b_Kb_iP_{k,r-2}^{(1,2,2)}(\lambda_i,\lambda_i^+,\lambda_i^-)}|_{e_i}
=-\abs{\nabla\lambda_i}b_i^2P_{r-2-k}^{(2k+5,1)}(-1)P_k^{(2,2)}(\lambda_i^+-\lambda^-),
\]
For $j=1,\cdots,r-2$. Using the relation~\eqref{eq:orthJacobi1}, and noting that $b_i=\lam_i^+\lam_i^{-}
=\lam_i^+(1-\lam_i^+)$ on $e_i$, we obtain
\begin{align*}
&\negint_{e_i}\diff{}{\boldsymbol n}\Lr{b_Kb_iP_{k,r-2}^{(1,2,2)}(\lambda_i,\lambda_i^+,\lambda_i^-)}P_j^{(2,2)}(\lambda_i^+-\lambda^-)\dsx\\
=&-\abs{\nabla\lambda_i}P_{r-2-k}^{(2k+5,1)}(-1)\int_0^1\Lr{\lam_i^+(1-\lam_i^+)}^2P_k^{(2,2)}(2\lambda_i^+-1)P_j^{(2,2)}(2\lambda_i^+-1)\mathrm{d}\lambda_i^+\\
=&-\dfrac{\abs{\nabla\lambda_i}}{32}P_{r-2-k}^{(2k+5,1)}(-1)\int_{-1}^1(1-t)^2(1+t)^2P_k^{(2,2)}(t)P_j^{(2,2)}(t)\md t\\
=&-\dfrac{\abs{\nabla\lambda_i}}{32}P_{r-2-k}^{(2k+5,1)}(-1)h_k^{(2,2)}\delta_{jk}.
\end{align*}
This gives 
\[
a_{k,r-2}=-\dfrac{32}{\abs{\nabla\lambda_i}P_{r-2-k}^{(2k+5,1)}(-1)h_k^{(2,2)}}.
\]

By~\cite{Szego:1975}*{Eq. (4.1.1),(4.1.3)}, we obtain
\(
P_{r-2-k}^{(2k+5,1)}(-1)=(-1)^{r-k}(r-k-1).
\)
Using~\eqref{eq:squaremoment1}, we obtain
\[
h_k^{(2,2)}=\dfrac{32(k+1)(k+2)}{(k+3)(k+4)(2k+5)}.
\]
A combination of the above three identities leads to~\eqref{eq:ak}.
\end{proof}

Next we list the shape functions for the elements of low-order.
\begin{example}
The bubble space $b_K\sum_{i=1}^3b_iQ_i^{r-2}(K)$ for the lowest-order $r=2$ is
\[
b_K\textup{span}\set{b_i}{i=1,2,3}.
\]
The shape functions associated with $\negint_{e_i}\partial_{\boldsymbol n}v\dsx$ is
\(-\dfrac{30}{\abs{\nabla\lambda_i}}b_Kb_i.\)

The bubble space $b_K\sum_{i=1}^3b_iQ_i^{r-2}(K)$ for the case $r=3$ is
\[b_K\textup{span}\set{b_i(3\lambda_i-\lambda_i^+-\lambda_i^-),b_i(\lambda_i^+-\lambda_i^-)}{i=1,2,3}.\]
The shape functions associated with $\negint_{e_i}\partial_{\boldsymbol n}v\dsx$ is
\[
\dfrac{30}{\abs{\nabla\lambda_i}}b_Kb_i(3\lambda_i-\lambda_i^+-\lambda_i^-).
\]
The shape functions associated with $3\negint_{e_i}\partial_{\boldsymbol n}v(\lambda_i^+-\lambda_i^-)\dsx$ is
\[
-\dfrac{70}{\abs{\nabla\lambda_i}}b_Kb_i(\lambda_i^+-\lambda_i^-).
\]
\end{example}

\begin{lemma}
The space $R^{r-2}(K)$ takes the forms
\[
R^{r-2}(K)=\textup{span}\set{P_{k,r-2}^{(2,2,2)}(\lambda_1,\lambda_2,\lambda_3)}{0\le k\le r-2}.
\]
\end{lemma}

\begin{proof}
For any $v\in R^{r-2}(K)\subset\mathbb{P}_{r-2}(K)$, we expand $v$ into
\[
v=\sum_{0\le k\le n\le r-2}a_{kn}P_{k,n}^{(2,2,2)}(\lambda_1,\lambda_2,\lambda_3)
\]
with $a_{kn}$ to be determined later on. Using the above representation, we may write the constraint in the definition of the space~\eqref{eq:bubble2} as: For all $q\in\mathbb{P}_{r-3}(K)$,
\[
\sum_{0\le k\le n\le r-2}a_{kn}\negint_Kb_K^2P_{k,n}^{(2,2,2)}(\lambda_1,\lambda_2,\lambda_3)q\mathrm{d}x=0.
\]
We substitute $q=P_{j,m}^{(2,2,2)}(\lambda_1,\lambda_2,\lambda_3)$ for $0\le j\le m\le r-3$ into the above equation. Using the orthogonal relation~\eqref{eq:orthJacobi2}, we may obtain $a_{jm}=0$ for $0\le j\le m\le r-3$.
This completes the proof.
\end{proof}

Motivated by the above lemma, we change the definition of DoFs for $b_K^2R^{r-2}(K)$ from $\int_Kvq\dx$ for any $q\in\mb{P}_{r-2}(K)\setminus\mb{P}_{r-3}(K)$ to
\[
\negint_KvP_{k,r-2}^{(2,2,2)}(\lambda_1,\lambda_2,\lambda_3)\dx, \quad k=0,\cdots, r-2.
\]

\begin{lemma}
The shape functions for the bubble space $b_K^2R^{r-2}(K)$ associated with the above definition of DoFs are
\[
a_{k,r-2}b_K^2P_{k,r-2}^{(2,2,2)}(\lambda_1,\lambda_2,\lambda_3)
\]
with
\[
a_{k,r-2}=\dfrac{(2r+4)(2k+5)(r+k+4)(r+k+5)(k+3)(k+4)}{2(r-k)(r-1-k)(k+1)(k+2)}\qquad k=0,\cdots,r-2.
\]
\end{lemma}

\begin{proof}
For $j=1,\cdots,r-2$, we obtain
\[\negint_Kb_K^2P_{k,r-2}^{(2,2,2)}(\lambda_1,\lambda_2,\lambda_3)P_{j,r-2}^{(2,2,2)}(\lambda_1,\lambda_2,\lambda_3)\dx=2h_{k,r-2}^{(2,2,2)}\delta_{jk},
\]
which gives
\[
a_{k,r-2}=\dfrac1{2h_{k,r-2}^{(2,2,2)}}.
\]
Using~\eqref{eq:squaremoment2}, we obtain
\[
h_{k,r-2}^{(2,2,2)}=\dfrac{(r-k)(r-1-k)(k+1)(k+2)}{(2r+4)(2k+5)(r+k+4)(r+k+5)(k+3)(k+4)}.
\]
These give the simplified expression of $a_{k,r-2}$.
\end{proof}

According to the definition of bubble space, we may have
\begin{example}
The bubble space $b_K^2R^{r-2}(K)$ for the lowest-order case $r=2$ is span$\{b_K^2\}$.

The shape functions associated with $\negint_{K}v\dx$ is $2520b_K^2$.

The bubble space $b_K^2R^{r-2}(K)$ for the case $r=3$ is $b_K^2$span$\{2\lam_1-\lam_2-\lam_3,\lambda_2-\lambda_3\}$.

The shape functions associated with $3\,\negint_Kv(2\lam_1-\lam_2-\lam_3)\dx$ is $4200b_K^2(2\lam_1-\lam_2-\lam_3)$.

The shape functions associated with $3\,\negint_Kv(\lambda_2-\lambda_3)\dx$ is $12600b_K^2(\lambda_2-\lambda_3)$.
\end{example}
\begin{remark}
Based on the above results, we give the details for the element of the lowest-order, i.e., $r=2$, which have been used in the numerical examples. The local finite element space
\[
P_K=\mb{P}_2(K)+b_K\sum_{i=1}^3\text{span}\{b_i\}+\text{span}\{b_K^2\},
\]
and DoFs
\[
\Sigma_K=\set{v(a_i),\negint_{e_i}v\dsx,\negint_{e_i}\pa_{\boldsymbol n}v\dsx,\negint_Kv\dx}{i=1,2,3}.
\]
The shape functions associated with $\{v(a_i)\}_{i=1,2,3}$ are
\[
\phi_i=\lambda_i(3\lambda_i-2)+30b_K\Lr{2b_i+\sum_{j\neq i}\dfrac{\na\lam_i\cdot\na\lam_j}{\abs{\na\lam_j}^2}b_j(4\lam_j-1)+6b_K}.
\]
The shape functions associated with $\{\negint_{e_i}v\md\sigma(\boldsymbol x)\}_{i=1,2,3}$ are
\[
\varphi_i=6b_i+90b_K\Lr{b_i-\sum_{j\neq i}b_j-10b_K}.
\]
The shape functions associated with $\{\negint_{e_i}\partial_{\boldsymbol n}v\md\sigma(\boldsymbol x)\}_{i=1,2,3}$ are
\[
\psi_i=\dfrac{30}{\abs{\na\lam_i}}b_Kb_i(4\lam_i-1).
\]
The shape functions associated with $\negint_K v\dx$ is $\phi_0=2520b_K^2$.
\end{remark}

%
\section{The Mixed Finite Elements Approximation}
In this part, we construct stable finite element pairs to approximate~\eqref{eq:mix}. We ignore the influence of the curved boundary for error estimates for brevity. 
Define
\[
X_h{:}=\set{v\in H_0^1(\Omega)}{v|_K\in P_K\text{ for all }K\in\T_h,\quad \int_e\jump{\partial_{\boldsymbol n}v}q\dsx=0\text{ for all }e\in\mc{E}_h,q\in\mb{P}_{r-2}(e)},
\]
and $V_h{:}=[X_h]^2$. Let $P_h\subset P$ be the continuous Lagrangian finite element of order $r-1$. We shall prove a uniform discrete B-B inequality for the pair $(V_h,P_h)$. 

The following rescaled trace inequality will be used later on: There exists $C$ independent of $h_K$ such that
\begin{equation}\label{eq:trace}
\nm{v}{L^2(\pa K)}\le C\Lr{h_K^{-1/2}\nm{v}{L^2(K)}+\nm{v}{L^2(K)}^{1/2}\nm{\na\boldsymbol v}{L^2(K)}^{1/2}}.
\end{equation}
This inequality may be found in~\cite{BrennerScott:2008}.
\subsection{The mixed finite element approximation}
We define the mixed finite element approximation problem as follows. Find $\boldsymbol u_h\in V_h$ and $p_h\in P_h$ such that
\begin{equation}\label{eq:mixap}
\left\{\begin{aligned}
a_{\iota,h}(\boldsymbol u_h,\boldsymbol v)+b_{\iota,h}(\boldsymbol v,p_h)&=(\boldsymbol f,\boldsymbol v)\qquad&&\text{for all\quad}\boldsymbol v\in V_h,\\
b_{\iota,h}(\boldsymbol u_h,q)-\lam^{-1}c_\iota(p_h,q)&=0\qquad&&\text{for all\quad}q\in P_h,
\end{aligned}\right.
\end{equation}
where 
\[
\begin{aligned}
a_{\iota,h}(\boldsymbol v,\boldsymbol w){:}&=2\mu\lr{(\boldsymbol\eps(\boldsymbol v),\boldsymbol\eps(\boldsymbol w))+\iota^2(\na_h\boldsymbol\eps(\boldsymbol v),\na_h\boldsymbol\eps(\boldsymbol w))}\quad &&\text{for all\,}\boldsymbol v,\boldsymbol w\in V_h,\\
b_{\iota,h}(\boldsymbol v,q){:}&=(\divop\boldsymbol v, q)+\iota^2(\na_h\divop\boldsymbol v,\na q)\quad &&\text{for all\,}\boldsymbol v\in V_h, q\in P_h.
\end{aligned}
\]

Note that $V_h\not\subset V$, and this is a nonconforming method, we introduce the broken norm
\[
\nm{\nabla\boldsymbol v}{\iota,h}{:}=\nm{\na\boldsymbol v}{L^2}+\iota\nm{\na_h^2\boldsymbol v}{L^2}\qquad\text{for all}\quad \boldsymbol v\in V_h
\]
Due to the continuity of $\boldsymbol v$, $\nm{\na\boldsymbol v}{\iota,h}$ is a norm over $V_h$.%

The following broken Korn's inequality was proved in~\cite{LiMingWang:2021}*{Theorem 2}:
\[
\nm{\na_h\boldsymbol\eps(\boldsymbol v)}{L^2}\ge \Lr{1-1/\sqrt2}\nm{\na_h^2\boldsymbol v}{L^2},
\]
which together with the first Korn's inequality~\eqref{eq:1stkorn} gives 
\begin{equation}\label{eq:diskorn}
a_{\iota,h}(\boldsymbol v,\boldsymbol v)\ge\dfrac\mu2\nm{\na\boldsymbol v}{\iota,h}^2\qquad\text{for all }\boldsymbol v\in V_h.
\end{equation}%

It remains to prove the discrete B-B inequality for the pair $(V_h,P_h)$. To this end, we construct a Fortin operator that is uniformly stable in the weighted norm $\nm{\na\cdot}{\iota,h}$~\cite{Winther:2013}. The key is to construct different Fortin operators for $\iota/h$ in different regimes.

Firstly we define an interpolation operator $\varPi_h:V\to V_h$ by $\varPi_h|_K{:}=\varPi_K=[\pi_K]^2$, which satisfies
\begin{lemma}
For all $\boldsymbol v\in V$, there holds
\begin{equation}\label{eq:inter2}
b_{\iota,h}(\varPi_h\boldsymbol v,p)=b_\iota(\boldsymbol v,p)\qquad\text{for all\;}p\in P_h.
\end{equation}
\end{lemma}

\begin{proof}
Using the fact that $\varPi_h\boldsymbol v\in V_h\subset [H_0^1(\Omega)]^2$, an integration by parts gives
\begin{equation}\label{eq:dividen1}
\int_{\Omega}\divop(\boldsymbol v-\varPi_h\boldsymbol v)p\dx
=-\sum_{K\in\mathcal{T}_h}\int_K(\boldsymbol v-\varPi_K\boldsymbol v)\cdot\nabla p\dx=0,
\end{equation}
where we have used the identity~\eqref{eq:inter1}$_4$ in the last step.

Next,  integration by parts yields
\[
\int_{\Omega}\nabla\divop(\boldsymbol v-\varPi_h\boldsymbol v)\cdot\nabla p\dx
=\sum_{K\in\mathcal{T}_h}\int_{\partial K}\divop(\boldsymbol v-\varPi_K\boldsymbol v)\pa_{\boldsymbol n} p\dsx-\sum_{K\in\mathcal{T}_h}\int_K\divop(\boldsymbol v-\varPi_K\boldsymbol v)\Delta p\dx=0.
\]
The first term vanishes because $\partial_j=t_j\partial_{\boldsymbol t}+n_j\partial_{\boldsymbol n}$ for components $j=1,2$, and using~\eqref{eq:inter1}$_2$, we obtain that for each edge $e\in\pa K$, 
\[
\int_et_j\diff{}{\boldsymbol t}(v_j-\pi_Kv_j)\pa_{\boldsymbol n} p\dsx=-\int_et_j(v_j-\pi_Kv_j)\diff{^2p}{t\partial n}\dsx=0.
\]
Using~\eqref{eq:inter1}$_3$, we obtain
\[
\int_en_j\diff{}{\boldsymbol n}(v_j-\pi_Kv_j)\pa_{\boldsymbol n} p\dsx=0.
\]

While the second term vanishes because
\[
\int_K\divop(\boldsymbol v-\varPi_K\boldsymbol v)\Delta p\dx=\int_{\partial K}(\boldsymbol v-\varPi_K\boldsymbol v)\cdot\boldsymbol n\Delta p\dsx-\int_K(\boldsymbol v-\varPi_K\boldsymbol v)\cdot\nabla\Delta p\dx=0,
\]
where we have used~\eqref{eq:inter1}$_2$ and~\eqref{eq:inter1}$_4$.
\end{proof}

The operator $\varPi_h$ is not H$^1-$bounded by~\eqref{eq:localerr}, and we construct an H$^1-$bounded regularized interpolation operator as follows. 
\begin{lemma}
There exists an operator $I_h:\,V\mapsto V_h$ satisfying
\begin{equation}\label{eq:dividen2}
\int_{\Omega}\divop(\boldsymbol v-I_h\boldsymbol v)p\dx=0\qquad\text{for all\quad}p\in P_h,
\end{equation}
and if $\boldsymbol v\in V\cap [H^s(\Om)]^2$ with $1\le s\le r+1$, then
\begin{equation}\label{eq:globalerr}
\nm{\na^j_h(\boldsymbol v-I_h\boldsymbol v)}{L^2}\le Ch^{s-j}\nm{\na^s\boldsymbol v}{L^2}\qquad 0\le j\le s.
\end{equation}
\end{lemma}

The construction of $I_h$ is based on a regularized interpolation operator in~\cite{GuzmanLeykekhmanNeilan:2012} and the standard construction of the Fortin operator~\cite{Fortin:1977}. The operator $I_h$ is also well-defined for functions in $[H^2(\Om)\cap H_0^1(\Om)]^2$.
\begin{proof}
Define $I_h\,:V\mapsto V_h$ with $I_h{:}=[\varPi_1]^2$ and
\[
\varPi_1{:}=\varPi_0(I-\varPi_2)+\varPi_2,
\]
where the regularized interpolation operator $\varPi_2:H_0^2(\Omega)\mapsto X_h$ was constructed in~\cite{GuzmanLeykekhmanNeilan:2012}*{Lemma 2}, which satisfies
\begin{equation}\label{eq:guzmanerr}
\nm{\na^j_h(v-\varPi_2v)}{L^2}\le Ch^{s-j}\nm{\na^sv}{L^2},\qquad 1\le s\le r+1, 0\le j\le s.
\end{equation}
The operator $\varPi_0\,:H_0^1(\Om)\mapsto X_h$ is defined for any element $K\in\mc{T}_h$ as
\[
\left\{\begin{aligned}
\varPi_0 v(a)&=0\quad\text{for all vertices}\,a,\\
\int_e\varPi_0vq\dsx&=0\quad\forall q\in\mb{P}_{r-2}(e)\quad\text{for all edges}\,e, \\
\int_e\pa_{\boldsymbol n}\varPi_0vq\dsx&=0\quad\forall q\in\mb{P}_{r-2}(e)\quad\text{for all edges}\,e,\\
\int_K\varPi_0vq\dx&=\int_Kvq\dx\quad\text{for all\quad} q\in\mb{P}_{r-2}(K).
\end{aligned}\right.
\]

On each element $K$, we have $(\varPi_0v)|_K \in H_0^1(K)$ for any $v\in H_0^1(\Om)$, and hence $\varPi_0 v\in X_h$. A standard scaling argument gives
\[
\nm{\varPi_0v}{L^2}\le C\nm{v}{L^2}\qquad\text{for all }v\in H_0^1(\Omega).
\]

For any $\boldsymbol v\in V$, using the fact that $I_h\boldsymbol v\in V_h$ and $p\in P_h$, an integration by parts gives
\begin{align*}
\int_{\Omega}\divop(\boldsymbol v-I_h\boldsymbol v)p\dx&=-\sum_{K\in\mathcal{T}_h}\int_K(\boldsymbol v-I_h\boldsymbol v)\cdot\nabla p\dx+\sum_{K\in\mathcal{T}_h}
\int_{\pa K}(\boldsymbol v-I_h\boldsymbol v)\cdot\boldsymbol n p\dsx\\
&=-\sum_{K\in\mathcal{T}_h}\int_K(\boldsymbol v-I_h\boldsymbol v)\cdot\nabla p\dx=-\sum_{K\in\mathcal{T}_h}\int_K(I-\varPi_0)(I-\varPi_2)v\cdot\nabla p\dx\\
&=0,
\end{align*}
where we have used the last property of $\varPi_0$. This gives~\eqref{eq:dividen2}.

Using~\eqref{eq:guzmanerr}, the L$^2$-stability of $\varPi_0$ and the inverse inequality, we obtain
\begin{align*}
\nm{\na^j_h(\boldsymbol v-\varPi_1\boldsymbol v)}{L^2}\le &\nm{\na^j_h(\boldsymbol v-\varPi_2\boldsymbol v)}{L^2}+\nm{\na^j_h\varPi_0(\boldsymbol v-\varPi_2\boldsymbol v)}{L^2}\\
\le&Ch^{s-j}\nm{\na^s\boldsymbol v}{L^2}+Ch^{-j}\nm{\boldsymbol v-\varPi_2\boldsymbol v}{L^2}\\
\le&Ch^{s-j}\nm{\na^s\boldsymbol v}{L^2}.
\end{align*}
This implies~\eqref{eq:globalerr} and completes the proof.
\end{proof}

We are ready to prove the following discrete B-B inequality.
\begin{theorem}
There exists $\beta$ independent of $\iota$ and $h$, such that 
\begin{equation}\label{eq:disbb}
\sup_{\boldsymbol v\in V_h}\dfrac{b_{\iota,h}(\boldsymbol v,p)}{\nm{\na\boldsymbol v}{\iota,h}}\ge\beta\inm{p}\qquad\text{for all\quad}p\in P_h.
\end{equation}
\end{theorem}

\begin{proof}
Using~\eqref{eq:estdiv1}, for any $p\in P_h\subset P$, there exists $\boldsymbol v_0\in V$ such that
\[
b_\iota(\boldsymbol v_0,p)=\inm{p}^2\qquad\text{and\quad} \inm{\na\boldsymbol v_0}\le C\inm{p}.
\]

First, we consider the case $\iota/h\le \gamma$ with $\gamma$ to be determined later on. By~\eqref{eq:globalerr}, we obtain
\[
\nm{\na I_h\boldsymbol v_0}{\iota,h}\le\inm{\na\boldsymbol v_0}+\nm{\na(\boldsymbol v_0-I_h\boldsymbol v_0)}{\iota,h}\le C\inm{\na\boldsymbol v_0},
\]
and
\[
\nm{\na_h\divop(\boldsymbol v_0-I_h\boldsymbol v_0)}{L^2}\le C\nm{\na^2\boldsymbol v_0}{L^2}\le C\nm{\na p}{L^2}.
\]
Combining the above inequality and using the inverse inequality for any $p\in P_h$, we obtain
\[
\iota^2\abs{(\na_h\divop(\boldsymbol v_0-I_h\boldsymbol v_0),\na p)}\le C\iota^2\nm{\na p}{L^2}^2\le C_{\ast}(\iota/h)^2\nm{p}{L^2}^2\le \gamma^2C_{\ast}\inm{p}^2.
\]
Fix $\gamma$ such that $\gamma^2C_{\ast}<1$, we obtain
\[
b_{\iota,h}(I_h\boldsymbol v_0,p)=b_{\iota}(\boldsymbol v_0,p)-\iota^2(\na_h\divop(\boldsymbol v_0-I_h\boldsymbol v_0),\na p)\ge (1-\gamma^2C_{\ast})\inm{p}^2.
\]
This gives 
\begin{equation}\label{eq:disbb1}
\sup_{v\in V_h}\dfrac{b_{\iota,h}(\boldsymbol v,p)}{\nm{\na\boldsymbol v}{\iota,h}}\ge\dfrac{b_{\iota,h}(I_h\boldsymbol v_0,p)}{\nm{\na I_hv_0}{\iota,h}}\ge\dfrac{1-\gamma^2C_*}{C}\inm{p}.
\end{equation}

Next, if $\iota/h>\gamma$, then we use~\eqref{eq:localerr} and obtain
\[
\nm{\na(\boldsymbol v_0-\varPi_h\boldsymbol v_0)}{L^2}\le Ch\nm{\na^2\boldsymbol v_0}{L^2},\quad\text{and}\quad
\nm{\na^2_h(\boldsymbol v_0-\varPi_h\boldsymbol v_0)}{L^2}\le C\nm{\na^2\boldsymbol v_0}{L^2}.
\]
Therefore,
\[
\nm{\na(\boldsymbol v_0-\varPi_h\boldsymbol v_0)}{\iota,h}\le C(h+\iota)\nm{\na^2\boldsymbol v_0}{L^2}\le C(1+h/\iota)\inm{\na\boldsymbol v_0}\le C (1+1/\gamma)\inm{\na\boldsymbol v_0}.
\]
Hence,
\[
\nm{\na\varPi_h\boldsymbol v_0}{\iota,h}\le\inm{\na\boldsymbol v_0}+\nm{\na(\boldsymbol v_0-\varPi_h\boldsymbol v_0)}{\iota,h}\le C (2+1/\gamma)\inm{\na\boldsymbol v_0},
\]
which together with~\eqref{eq:inter2} gives
\begin{equation}\label{eq:disbb2}
\sup_{v\in V_h}\dfrac{b_{\iota,h}(\boldsymbol v,p)}{\nm{\na\boldsymbol v}{\iota,h}}\ge\dfrac{b_{\iota,h}(\varPi_h\boldsymbol v_0,p)}{\nm{\na\varPi_h\boldsymbol v_0}{\iota,h}}
=\dfrac{b_{\iota,h}(\boldsymbol v_0,p)}{\nm{\na\varPi_h\boldsymbol v_0}{\iota,h}}
=\dfrac{\inm{p}^2}{\nm{\na\varPi_h\boldsymbol v_0}{\iota,h}}\ge\dfrac{\gamma}{C(1+2\gamma)}\inm{p}.
\end{equation}

A combination of~\eqref{eq:disbb1} and~\eqref{eq:disbb2} shows that~\eqref{eq:disbb} holds true with $\beta$ independent of $\iota$ and $h$.
\end{proof}

The well-posedness of the mixed approximation problem~\eqref{eq:mixap} follows from the ellipticity of $a_{\iota,h}$ and the discrete B-B inequality of $b_{\iota,h}$. We are ready to derive the error estimate. 
\subsection{Error estimates}
To carry out the error estimate, we define the bilinear form $\mc{A}$ as
\[
\mc{A}(\boldsymbol v,q;\boldsymbol w,z){:}=a_{\iota,h}(\boldsymbol v,\boldsymbol w)+b_{\iota,h}(\boldsymbol w,q)+b_{\iota,h}(\boldsymbol v,z)-\lam^{-1}c_{\iota}(q,z)
\]
for all $\boldsymbol v,\boldsymbol w\in V_h$ and $q,z\in P_h$.

We prove the following inf-sup inequality for $\mc{A}$ with the aid of the discrete B-B inequality~\eqref{eq:disbb}.
\begin{lemma}\label{lema:infsup}
There exists $\al$ depending on $\mu$ and $\beta$ such that 
\begin{equation}\label{eq:infsup}
\inf_{(\boldsymbol v,q)\in V_h\x P_h}\sup_{(\boldsymbol w,z)\in V_h\x P_h}\dfrac{\mc{A}(\boldsymbol v,q;\boldsymbol w,z)}{\wnm{(\boldsymbol w,z)}\wnm{(\boldsymbol v,q)}}\ge\al,
\end{equation}
where $\wnm{(\boldsymbol w,z)}{:}=\nm{\na\boldsymbol w}{\iota,h}+\inm{z}+\lam^{-1/2}\inm{z}$ and $\beta$ has appeared in~\eqref{eq:disbb}.
\end{lemma}

\begin{proof}
Noting that $a_{\iota,h}$ is elliptic over $V_h$~\eqref{eq:diskorn} and the discrete B-B inequality for $b_{\iota,h}$ holds~\eqref{eq:disbb}, we obtain~\eqref{eq:infsup} by~\cite{Braess:1996}*{Theorem 2}.
\end{proof}

We are ready to prove error estimates. 
\begin{theorem}\label{thm:finalerr}
There exists $C$ independent of $\iota,\lambda$ and $h$ such that
\begin{equation}\label{eq:finalerr}
\wnm{(\boldsymbol u-\boldsymbol u_h,p-p_h)}\le C(h^r+\iota h^{r-1})(\nm{\boldsymbol u}{H^{r+1}}+\nm{p}{H^r}),
\end{equation}
and
\begin{equation}\label{eq:finalerr2}
  \wnm{(\boldsymbol u-\boldsymbol u_h,p-p_h)}\le Ch^{1/2}\nm{\boldsymbol f}{L^2}.
\end{equation}
\end{theorem}

\begin{proof}
Let $\boldsymbol v=\boldsymbol u_h-\boldsymbol u_I$ and $q=p_h-p_I$ with $\boldsymbol u_I\in V_h$ and $p_I\in P_h$, 
for any $\boldsymbol w\in V_h$ and $z\in P_h$,
\begin{align*}
\mc{A}(\boldsymbol v,q;\boldsymbol w,z)&=\mc{A}(\boldsymbol u_h,p_h;\boldsymbol w,z)-\mc{A}(\boldsymbol u,p;\boldsymbol w,z)+\mc{A}(\boldsymbol u-\boldsymbol u_I,p-p_I;\boldsymbol w,z)\\
&=(\boldsymbol f,\boldsymbol w)-\mc{A}(\boldsymbol u,p;\boldsymbol w,z)+\mc{A}(\boldsymbol u-\boldsymbol u_I,p-p_I;\boldsymbol w,z)\\
&=\mc{A}(\boldsymbol u-\boldsymbol u_I,p-p_I;\boldsymbol w,z)-\iota^2\sum_{e\in\mc{E}_h}\int_e(\partial_{\boldsymbol n}\boldsymbol\sigma\boldsymbol n)\cdot\jump{\pa_{\boldsymbol n}\boldsymbol w}\dsx.
\end{align*}

The boundedness of $\mc{A}$ yields
\[
\abs{\mc{A}(\boldsymbol u-\boldsymbol u_I,p-p_I;\boldsymbol w,z)}
\le\max(1,2\mu)\wnm{(\boldsymbol u-\boldsymbol u_I,p-p_I)}\wnm{(\boldsymbol w,z)}.
\]

Let $\boldsymbol u_I=\varPi_h\boldsymbol u$ be the interpolation of $\boldsymbol u$ and $p_I$ be the $r-1$ order Lagrangian interpolation of $p$, respectively.  The standard interpolation error estimates in~\eqref{eq:localerr} gives
\[
\wnm{(\boldsymbol u-\boldsymbol u_I,p-p_I)}\le C(h^r+\iota h^{r-1})(\nm{\boldsymbol u}{H^{r+1}}+\nm{p}{H^r}).
\]
Note that
\[
\int_e\jump{\pa_{\boldsymbol n}\boldsymbol w}q\dsx=0\qquad\text{for all }q\in\mb{P}_{r-2}(e).
\]
A standard estimate for the consistency error functional with trace inequality~\eqref{eq:trace} gives
\begin{align*}
\iota^2\labs{\sum_{e\in\mc{E}_h}\int_e(\partial_{\boldsymbol n}\boldsymbol\sigma\boldsymbol n)\cdot\jump{\pa_{\boldsymbol n}\boldsymbol w}\dsx}&\le C\iota^2h^{r-1}(\nm{\boldsymbol u}{H^{r+1}}+\nm{p}{H^r})\nm{\na_h^2\boldsymbol w}{L^2}\\
&\le C\iota h^{r-1}(\nm{\boldsymbol u}{H^{r+1}}+\nm{p}{H^r})\nm{\na\boldsymbol w}{\iota,h}.
\end{align*}
A combination of the above three inequalities, the discrete inf-sup condition~\eqref{eq:infsup} and the triangle inequalities gives~\eqref{eq:finalerr}. 

Next, let $\boldsymbol u_I=I_h\boldsymbol u$ and let $p_I$ be the Cl\'ement interpolation~\cite{Clement:1975} of $p$, respectively. The interpolation error~\eqref{eq:globalerr} and the error estimates for the Cl\'ement interpolation give
\[
  \wnm{(\boldsymbol u-\boldsymbol u_I,p-p_I)}\le Ch^{1/2}\lr{\nm{\boldsymbol u}{H^{3/2}}+\nm{p}{H^{1/2}}+\iota(\nm{\boldsymbol u}{H^{5/2}}+\nm{p}{H^{3/2}})}\le Ch^{1/2}\nm{\boldsymbol f}{L^2},
\]
where we have used~\eqref{eq:h32reg} and~\eqref{eq:h52reg} in the last step. 

Using the trace inequality~\eqref{eq:trace}, we bound the consistency error functional as
\begin{align*}
\iota^2\labs{\sum_{e\in\mc{E}_h}\int_e(\partial_{\boldsymbol n}\boldsymbol\sigma\boldsymbol n)\cdot\jump{\pa_{\boldsymbol n}\boldsymbol w}\dsx}&\le C\iota^2h^{1/2}(\nm{\boldsymbol u}{H^2}+\nm{p}{H^1})^{1/2}(\nm{\boldsymbol u}{H^3}+\nm{p}{H^2})^{1/2}\nm{\na_h^2\boldsymbol w}{L^2}\\
&\le Ch^{1/2}\nm{\boldsymbol f}{L^2}\nm{\na\boldsymbol w}{\iota,h},
\end{align*}
where we have used~\eqref{eq:h2reg} in the last step. 

Combining these inequalities, the discrete inf-sup condition~\eqref{eq:infsup} and the triangle inequalities gives~\eqref{eq:finalerr2}. 
\end{proof}

%
\begin{coro}
There exists $C$ independent of $\iota,\lambda$ and $h$ such that
\begin{equation}\label{eq:finalerr3}
\wnm{(\boldsymbol u_0-\boldsymbol u_h,p_0-p_h)}\le C(\iota^{1/2}+h^{1/2})\nm{\boldsymbol f}{L^2},
\end{equation}
where $\boldsymbol u_0$ is the solution of~\eqref{eq:elas}, and $p_0=\lambda\divop\boldsymbol u_0$.
\end{coro}

\begin{proof}
A combination of Theorem~\ref{thm:reg}, Theorem~\ref{thm:finalerr}, and the triangle inequality gives~\eqref{eq:finalerr3}.
\end{proof}

\section{Numerical Examples}\label{sec:numerical}
In this part, we report the numerical performance for the proposed element of the lowest-order, i.e., $r=2$. 
We test the accuracy and robustness of the element pair for the nearly incompressible materials. All examples are carried out on the nonuniform mesh. We are interested in the case when the Poisson's ratio $\nu$ is close to $0.5$ and we report the relative errors $\nm{\na(\boldsymbol u-\boldsymbol u_h)}{\iota,h}/\inm{\na\boldsymbol u}$ and the rates of convergence.

We let $\Omega=(0,1)^2$, and set Young's modulus $E=1$. The Lam\'e constants are determined by
\[
\lambda=\dfrac{E\nu}{(1+\nu)(1-2\nu)}, \quad \mu=\dfrac{E}{2(1+\nu)}.
\]
We set $\nu=0.3$ for the ordinary cases, hence $\lambda=0.5769$, and $\mu=0.3846$, and we set $\nu=0.4999$ for the nearly incompressible materials, hence $\lambda=\text{1.6664e3}$, and $\mu=0.3334$.
\subsection{The first example}
We test the performance of the element pair by solving a completely incompressible problem, which means $\divop\boldsymbol u=0$. Let $\boldsymbol u=(u_1,u_2)$ with
\[
u_1=-\sin^3(\pi x)\sin(2\pi y)\sin(\pi y),\quad u_2=\sin(2\pi x)\sin(\pi x)\sin^3(\pi y).
\]
Therefore $\divop\boldsymbol u=0$, and $\boldsymbol f$ is independent of $\lambda$.
\begin{table}[htbp]\centering\caption{Relative errors and convergence rates for the 1st example}~\label{tab:case1Guzman}\begin{tabular}{ccccc}
\hline
$\iota\backslash h$ & 1/8 & 1/16 & 1/32 & 1/64\\
\hline
\multicolumn{5}{c}{$\nu=0.3,\lambda=0.5769,\mu=0.3846$}\\
\hline
1e+00 & 2.592e-01 & 1.333e-01 & 6.519e-02 & 3.246e-02\\
rate & & 0.96 & 1.03 & 1.01\\
1e-06 & 4.252e-02 & 1.159e-02 & 2.784e-03 & 6.918e-04\\
rate & & 1.88 & 2.06 & 2.01\\
\hline
\multicolumn{5}{c}{$\nu=0.4999,\lambda=\text{1.6664e3}, \mu=0.3334$}\\
\hline
1e+00 & 2.592e-01 & 1.333e-01 & 6.519e-02 & 3.246e-02\\
rate & & 0.96 & 1.03 & 1.01\\
1e-06 & 4.252e-02 & 1.159e-02 & 2.784e-03 & 6.918e-04\\
rate & & 1.88 & 2.06 & 2.01\\
\hline
\end{tabular}\end{table}
%

In view of Table~\ref{tab:case1Guzman}, the optimal rates of convergence are observed with the completely incompressible media, which is consistent with the error bound~\eqref{eq:finalerr}.
%
\subsection{The second example}
This example is motivated by~\cite{Wihler:2006}, which admits a singular solution. The exact solution $\boldsymbol u=(u_1,u_2)$ expressed in the polar coordinates as
\[
u_1=u_{\rho}(\rho,\theta)\cos\theta-u_{\theta}(\rho,\theta)\sin\theta,\quad  
u_2=u_{\rho}(\rho,\theta)\sin\theta+u_{\theta}(\rho,\theta)\cos\theta,
\]
where
\begin{align*}
u_{\rho}&=\dfrac1{2\mu}\rho^{\alpha}\lr{-(\alpha+1)\cos((\alpha+1)\theta)+(C_2-(\alpha+1))C_1\cos((\alpha-1)\theta)},\\
u_{\theta}&=\dfrac1{2\mu}\rho^{\alpha}\lr{(\alpha+1)\sin((\alpha+1)\theta)+(C_2+\alpha-1)C_1\sin((\alpha-1)\theta)},
\end{align*}
and $\alpha=1.5, \omega=3\pi/4$, 
\[
C_1=-\dfrac{\cos((\alpha+1)\omega)}{\cos((\alpha-1)\omega)}\quad\text{and}\quad C_2=\dfrac{2(\lambda+2\mu)}{\lambda+\mu}.
\]

It may be verified that $\boldsymbol u\in [H^{5/2-\varepsilon}(\Omega)]^2$ for a small number $\varepsilon>0$. A direct calculation gives that $\boldsymbol f\equiv\boldsymbol 0$, while it is nearly incompressible because
\[
\divop\boldsymbol u=-\dfrac{3(1+\sqrt2)}{\lambda+\mu}\rho^{1/2}\cos(\theta/2).
\]
%
\begin{table}[htbp]\centering\caption{Relative errors and convergence rates for the 2nd example}~\label{tab:case2Guzman}\begin{tabular}{ccccc}
\hline
$\iota\backslash h$ & 1/8 & 1/16 & 1/32 & 1/64\\
\hline
\multicolumn{5}{c}{$\nu=0.3,\lambda=0.5769,\mu=0.3846$}\\
\hline
1e+00 & 1.062e-01 & 7.554e-02 & 5.355e-02 & 3.792e-02\\
rate & & 0.49 & 0.50 & 0.50\\
1e-06 & 2.809e-03 & 1.001e-03 & 3.549e-04 & 1.257e-04\\
rate & & 1.49 & 1.50 & 1.50\\
\hline
\multicolumn{5}{c}{$\nu=0.4999,\lambda=\text{1.6664e3}, \mu=0.3334$}\\
\hline
1e+00 & 1.149e-01 & 8.200e-02 & 5.824e-02 & 4.135e-02\\
rate & & 0.49 & 0.49 & 0.49\\
1e-06 & 4.399e-03 & 1.567e-03 & 5.558e-04 & 1.968e-04\\
rate & & 1.49 & 1.50 & 1.50\\
\hline
\end{tabular}\end{table}
%

It follows from Table~\ref{tab:case2Guzman} that the rates of convergence are sub-optimal. It is reasonable because the solution $\boldsymbol u$ is singular, which is similar to the results in~\cite{Wihler:2006}. The element pair is robust for the nearly incompressible materials.
\subsection{The third example}
In the last example, we test a problem with strong boundary layer effects. Such effects have been frequently observed in the stain elasticity model~\cites{Engel:2002, LiMingShi:2017, LiaoM:2019, LiMingWang:2021}. It is shown that the numerical solution converges to the solution of~\eqref{eq:elas} when $\iota\ll h$. 

When $\iota\to 0$, the boundary value problem~\eqref{eq:sgbvp} reduces to~\eqref{eq:elas}. Let $\boldsymbol u_0=(u_1^0,u_2^0)$ with
\[
u_1^0=-\sin^2(\pi x)\sin(2\pi y), \quad u_2^0=\sin(2\pi x)\sin^2(\pi y)
\]
be the solution of problem~\eqref{eq:elas}. The source term $\boldsymbol f$ is computed from~\eqref{eq:elas}. A direct calculation gives that $\divop\boldsymbol u_0=0$, and $\boldsymbol f$ is independent of $\lambda$. The exact solution $\boldsymbol u$ for~\eqref{eq:sgbvp} is unknown, while it has strong boundary layer effects. In this case, we take $\iota\ll h$, and report the relative error $\nm{\na(\boldsymbol u_0-\boldsymbol u_h)}{\iota,h}/\inm{\na\boldsymbol u_0}$.
\begin{table}[htbp]\centering\caption{Relative errors and convergence rates for the 3rd example}~\label{tab:case3Guzman}\begin{tabular}{ccccc}
\hline
$\iota\backslash h$ & 1/8 & 1/16 & 1/32 & 1/64\\
\hline
\multicolumn{5}{c}{$\nu=0.3,\lambda=0.5769,\mu=0.3846$}\\
\hline
1e-04 & 1.311e-01 & 8.966e-02 & 6.299e-02 & 4.476e-02\\
rate & & 0.55 & 0.51 & 0.49\\
1e-06 & 1.311e-01 & 8.960e-02 & 6.283e-02 & 4.432e-02\\
rate & & 0.55 & 0.51 & 0.50\\
\hline
\multicolumn{5}{c}{$\nu=0.4999,\lambda=\text{1.6664e3}, \mu=0.3334$}\\
\hline
1e-04 & 1.312e-01 & 8.968e-02 & 6.300e-02 & 4.476e-02\\
rate & & 0.55 & 0.51 & 0.49\\
1e-06 & 1.312e-01 & 8.963e-02 & 6.284e-02 & 4.432e-02\\
rate & & 0.55 & 0.51 & 0.50\\
\hline
\end{tabular}\end{table}
%

It follows from Table~\ref{tab:case3Guzman} that the rate of convergence for the element pair changes to $1/2$ because of the boundary layer effects, which is consistent with the theoretical result. The element is still robust when the solution has strong boundary layer effects in the nearly incompressible limit. 
\bibliographystyle{amsplain}
\bibliography{sg-9}
\end{document}